\documentclass{amsart}

\usepackage{fullpage, amssymb}

\newtheorem {theorem}              {Theorem}
\newtheorem {lemma}       [theorem]{Lemma}
\newtheorem {proposition} [theorem]{Proposition}
\newtheorem {conjecture}  [theorem]{Conjecture}
\newtheorem {question}    [theorem]{Question}
\newtheorem {axiom}       [theorem]{Axiom}

\theoremstyle{definition}
\newtheorem*{definition}           {Definition}

\begin{document}

\title{On minimal colorings without\\monochromatic solutions to a linear
  equation}

\author{Boris Alexeev}
\address{Massachusetts Institute of Technology\newline \indent
  Cambridge, MA 02139, USA}
\email{borisa@mit.edu}

\author{Jacob Fox}
\address{Massachusetts Institute of Technology\newline \indent
  Cambridge, MA 02139, USA}
\email{licht@mit.edu}

\author{Ron Graham}
\address{University of California, San Diego\newline \indent
  La Jolla, CA 92093, USA}
\email{graham@ucsd.edu}

\begin{abstract}
  For a ring $R$ and system $\mathcal{L}$ of linear homogeneous
  equations, we call a coloring of the nonzero elements of $R$
  \emph{minimal for $\mathcal{L}$} if there are no monochromatic
  solutions to $\mathcal{L}$ and the coloring uses as few colors as
  possible. For a rational number $q$ and positive integer $n$, let
  $E(q,n)$ denote the equation $\sum_{i=0}^{n-2} q^{i}x_i =
  q^{n-1}x_{n-1}$. We classify the minimal colorings of the nonzero
  rational numbers for each of the equations $E(q,3)$ with $q \in
  \{\frac{3}{2},2,3,4\}$, for $E(2,n)$ with $n \in \{3,4,5,6\}$, and for
  $x_1+x_2+x_3=4x_4$. These results lead to several open problems and
  conjectures on minimal colorings.
\end{abstract}

\maketitle

\section{Introduction}

The early developments in Ramsey theory focused mainly on partition
regularity of systems of linear equations. A system $\mathcal{L}$ of
linear homogeneous equations with coefficients in a ring $R$ is called
\emph{$r$-regular} over $R$ if, for every $r$-coloring of the nonzero
elements of $R$, there is a monochromatic solution to $\mathcal{L}$. A
system $\mathcal{L}$ of linear homogeneous equations is called
\emph{regular} over $R$ if it is $r$-regular over $R$ for all positive
integers $r$.

In 1916, Schur \cite{Sc} proved that the equation $x+y=z$ is regular
over $\mathbb{Z}$. In 1927, van der Waerden \cite{Wa} proved his
celebrated theorem that every finite coloring of the positive integers
contains arbitrarily long monochromatic arithmetic progressions. In his
1933 thesis, Rado \cite{Ra} generalized the theorems of Schur and van
der Waerden by classifying those systems of linear homogeneous equations
that are regular over $\mathbb{Z}$.  In particular, a linear homogeneous
equation with nonzero integer coefficients is regular over $\mathbb{Z}$
if and only if some nonempty subset of the coefficients sums to zero. In
1943, Rado \cite{Ra1} generalized the theorem further by classifying
those systems of linear equations that are regular over a subring of the
complex numbers. More recently, analogues of Rado's theorem have been
proven for abelian groups \cite{De1}, finite fields \cite{BeDeHi}, and
commutative rings \cite{BeDeHiLe}.

Some of the major remaining open problems on partition regularity
concern the properties of colorings that are free of monochromatic
solutions to a system of equations.

\begin{definition}
\label{mini} A coloring of the nonzero elements of a ring $R$ (or
more generally, a set of numbers $S$) is called \emph{minimal for a
  system $\mathcal{L}$ of linear homogeneous equations} if it is free of
monochromatic solutions to $\mathcal{L}$ and uses as few colors as
possible.
\end{definition}

Three basic questions arise for a given ring $R$ and system
$\mathcal{L}$ of linear homogeneous equations.

\begin{question}
What are the minimal colorings for $\mathcal{L}$?
\end{question}

\begin{question}\label{question:colors}
How many colors are used in a minimal coloring for $\mathcal{L}$?
\end{question}

\begin{question}\label{question:colorings}
How many minimal colorings, up to isomorphism, are there for
$\mathcal{L}$?
\end{question}

Rado made the following unresolved conjecture in his thesis on
Question~\ref{question:colors}.

\begin{conjecture} [Rado, \cite{Ra}] \label{RBC}
For all positive integers $m$ and $n$, there exists a positive integer
$k(m,n)$ such that if a system of $m$ linear equations in $n$ variables
is $k(m,n)$-regular over $\mathbb{Z}$, then the system is regular over
$\mathbb{Z}$.
\end{conjecture}

Conjecture~\ref{RBC} is commonly known as Rado's Boundedness Conjecture
\cite{HiLeSt}. Rado proved that Conjecture~\ref{RBC} is true if it is
true in the case when $m=1$, that is, for single linear equations. Rado
also settled his conjecture in the simple cases $n=1$ and
$n=2$. Kleitman and the second author \cite{FoKl} recently proved Rado's
Boundedness Conjecture for $n=3$. They proved that if a linear equation
in 3 variables is 36-regular over $\mathbb{Z}$, then it is regular over
$\mathbb{Z}$.

Rado also made the following conjecture in his thesis.

\begin{conjecture}[Rado, \cite{Ra}] \label{next}
For each positive integer $n$, there is a linear equation that is
$n$-regular over $\mathbb{Z}$ but not $(n+1)$-regular over $\mathbb{Z}$.
\end{conjecture}

While Conjecture~\ref{next} has remained open, Radoi\v{c}i\'c and the
second author found a family of linear homogeneous equations that they
conjecture verifies Conjecture~\ref{next}. For a rational number $q$ and
positive integer $n$, let $E(q,n)$ denote the equation
\[ x_0+qx_1+\dotsb+q^{n-2}x_{n-2}=q^{n-1}x_{n-1}. \]

\begin{definition}
For a nonzero rational number $q$ and prime number $p$, there is a
unique representation of $q$ as $q=p^va/b$ such that $a$, $b$, and $v$
are integers, $b$ is positive, $\gcd(a,b)=1$, and $p\nmid a,b$.  Define
$v_p(q)$ to be the integer $v$ and $w_p(q) \in \{1,\dotsc,p-1\}$ by
$w_p(q) \equiv ab^{-1} \pmod{p}$.
\end{definition}

For a prime $p$ and positive integer $n$, let $c_{p,n}\colon \mathbb{Q}
\setminus \{0\} \rightarrow \{0,1,\dotsc,n-1\}$ be the $n$-coloring of
the nonzero rational numbers defined by $c_{p,n}(q)\equiv v_p(q)
\pmod{n}$.

To avoid any possible confusion, we now define what it means for two
colorings of a set to be isomorphic.

\begin{definition}
  Two colorings $c_1\colon S \rightarrow C_1$ and $c_2\colon S
  \rightarrow C_2$ of the same set $S$ are \emph{isomorphic} if there is
  a bijection $\phi\colon C_2 \rightarrow C_1$ such that $c_1 =\phi
  \circ c_2$.
\end{definition}

Radoi\v ci\'c and the second author \cite{FoRa} proved that the
$n$-coloring $c_{p,n}$ is free of monochromatic solutions to
$E(p,n)$. Hence, the equation $E(p,n)$ is not $n$-regular over
$\mathbb{Z}$. They also conjecture that $E(2,n)$ is $(n-1)$-regular over
$\mathbb{Z}$, which would imply Conjecture~\ref{next}. We make the
following stronger conjecture.

\begin{conjecture} \label{stronger}
For $n>2$, $c_{2,n}$ is the only $n$-coloring, up to isomorphism, of the
nonzero rational numbers without a monochromatic solution to $E(2,n)$.
\end{conjecture}

In Section~\ref{section2}, we verify Conjecture~\ref{stronger} for $n=3$
and $n=4$. With a computer-generated proof, we have also verified
Conjecture~\ref{stronger} for $n=5$ and $n=6$, and that $E(2,7)$ is
$6$-regular. For brevity, we do not include these proofs. By the same
technique, it can be shown that $c_{3,3}$ is the only minimal coloring
of the nonzero rational numbers that is free of monochromatic solutions
to $E(3,3)$. We do, however, include a proof of the following result.

\begin{proposition} \label{proposition7}
The only $3$-colorings, up to isomorphism, of the nonzero rational
numbers without a monochromatic solution to $E(\frac{3}{2},3)$ are
$c_{2,3}$ and $c_{3,3}$.
\end{proposition}

As it pertains to Question~\ref{question:colorings}, the following
definition is natural.

\begin{definition}
\label{numb} For a system $\mathcal{L}$ of linear homogeneous
equations over a ring $R$, let $\Delta(\mathcal{L};R)$ denote the number
(as a cardinality) of minimal colorings, up to isomorphism, of the
nonzero elements of $R$ for $\mathcal{L}$.
\end{definition}

For example, we have $\Delta(E(2,n);\mathbb{Q})=1$ for $n \in
\{3,4,5,6\}$ and $\Delta(E(\frac{3}{2},3);\mathbb{Q})=2$.

The following conjecture would reduce the problem of finding
$\Delta(E(q,n);\mathbb{Q})$ to the case when $q$ is a prime power.

\begin{conjecture}
If $a$, $b$, and $n$ are integers, $n>2$, $a>1$, $\lvert b\rvert > 1$,
and $\gcd(a,b)=1$, then
\[ \Delta(E(a,n);\mathbb{Q})=\Delta(E(-a,n);\mathbb{Q}) \]
and
\[
\Delta(E(ab,n);\mathbb{Q})=\Delta(E(a/b,n);\mathbb{Q})=\Delta(E(a,n);\mathbb{Q})+\Delta(E(b,n);\mathbb{Q}). \]
\end{conjecture}

Let $\mathbb{Q}\setminus\{0\}=Q_0 \cup Q_1$ be the partition of the
nonzero rational numbers given by $q \in Q_0$ if and only if $v_2(q)$ is
even. Consider the $3$-coloring $c$ of $Q_0$ given by $c_0(q)=i$ if and
only if $v_2(q)\equiv 2i \pmod{6}$. For a permutation $\pi$ of the set
$\{0,1,2\}$, define the $3$-coloring $c_{\pi}$ of the nonzero rational
numbers by $c_{\pi}(q)=c(q)$ for $q \in Q_0$ and $c_{\pi(q)}=\pi(c(2q))$
for $q \in Q_1$. It is easy to check that the six colorings of the form
$c_{\pi}$ are each minimal for the equation $E(4,3)$. It can be shown
that these are all the minimal colorings of the nonzero rational numbers
for the equation $E(4,3)$, but we leave it out for brevity. This
construction can be easily generalized to give $(n!)^{r-1}$ different
$n$-colorings that are free of monochromatic solutions to $E(2^r,n)$. It
seems likely that these are the only minimal colorings for $E(2^r,n)$.

For each odd prime $p$ and integer $n > 2$, we can find
$(n!)^{\frac{p-3}{2}}$ different $n$-colorings of the nonzero rational
numbers that are free of monochromatic solutions to $E(p,n)$. These
$n$-colorings $c$ are given by the following two properties:
\begin{enumerate}
\item For every nonzero rational number $q$, we have $c(q)=c(p^jq)$
if and only if $j$ is a multiple of $n$. \item If
$v_p(q_1)=v_p(q_2)$ and $w_p(q_1) \equiv \pm w_p(q_2)
\pmod{p}$, then $c(q_1)=c(q_2)$.
\end{enumerate}

For an odd prime $p$, positive integers $n$ and $r$ with $n>2$, the
observations above can be generalized to construct a family of
$(n!)^{r+\frac{p-5}{2}}$ different $n$-colorings of the nonzero rational
numbers that are free of monochromatic solutions to $E(p^r,n)$. It seems
plausible, though we have shown little evidence to support it, that
these are all the minimal colorings for $E(p^r,n)$. This would imply
that $\Delta(E(p^r;n);\mathbb{Q})=(n!)^{r+\frac{p-5}{2}}$ for $n>2$, $p$
an odd prime, and $r$ a positive integer.

The total number of colorings of the nonzero rational numbers is
$2^{\aleph_0}$. Hence, for every system $\mathcal{L}$ of equations, we
have $\Delta(\mathcal{L};\mathbb{Q}) \leq 2^{\aleph_0}$. This upper
bound can be achieved, as the following proposition demonstrates.

\begin{proposition} \label{proposition9}
We have
\[ \Delta(x_1+x_2+x_3=4x_4;\mathbb{Q})=2^{\aleph_0}. \]
\end{proposition}
In fact, we classify all of the $2^{\aleph_0}$ minimal colorings for
$x_1+x_2+x_3=4x_4$ in Section~\ref{section3}.

For a prime number $p$, let $C_p$ be the $(p-1)$-coloring of the nonzero
rational numbers defined by $C_p(q) = w_p(q)$. For any set
$A=\{a_1,\dotsc,a_n\}$ such that no non-empty subset of $A$ sums to
zero, Rado proved that the equation $a_1x_1+\dotsb+a_nx_n=0$ is not
regular by showing that if $p$ is a sufficiently large prime number,
then the coloring $C_p$ is free of monochromatic solutions to
$a_1x_1+\dotsb+a_nx_n=0$. It turns out that the $4$-coloring $C_5$ is
minimal for the equation $x_1+x_2+x_3=4x_4$.

Let $\Pi_p=(\pi_n)_{n \in \mathbb{Z}}$ be a sequence of permutations of
the set $\{1,\dotsc,p-1\}$ of nonzero elements of $\mathbb{Z}_p$ such
that $\pi_0$ is the identity permutation. For each such sequence
$\Pi_p$, we define the coloring $c_{\Pi_p}$ of the nonzero rational
numbers by $c_{\Pi_p}(q)=\pi_{v_p(q)}(w_p(q))$. In particular, the
coloring $c_{\Pi_p}$ is the same as $C_p$ if $\pi_n$ is the identity
permutation for all integers $n$. It is a straightforward exercise to
show that each of the $(p-1)$-colorings $c_{\Pi_p}$ is free of
monochromatic solutions to the equation
$x_1+\dotsb+x_{p-2}=(p-1)x_{p-1}$. For $p=5$, we can say even more.

\begin{proposition} \label{proposition10}
Each of the $4$-colorings $c_{\Pi_5}$ is minimal for $x_1+x_2+x_3=4x_4$,
and there are no other minimal colorings for $x_1+x_2+x_3=4x_4$.
\end{proposition}

There are exactly $2^{\aleph_0}$ colorings of the form $c_{\Pi_5}$,
which establishes Proposition~\ref{proposition9}.

In all the examples above, either an equation has finitely many or
$2^{\aleph_0}$ minimal colorings of the nonzero rational numbers.  It
seems likely that this is always the case.

\begin{conjecture}
For every nonregular finite system $\mathcal{L}$ of linear homogeneous
equations, either $\Delta(\mathcal{L},\mathbb{Q})$ is finite or
$2^{\aleph_0}$. In particular, there is no $\mathcal{L}$ satisfying
$\Delta(\mathcal{L},\mathbb{Q})=\aleph_0$.
\end{conjecture}

\subsection{Minimal colorings over the real numbers}

When studying colorings of the real numbers, we must be careful about
which axioms we choose for set theory. In this subsection, we assume
that $q$ is any rational number other than $-1$, $0$, or $1$.

Radoi\v ci\'c and the second author \cite{FoRa} showed that it is
independent of Zermelo-Fraenkel (ZF) set theory that the equation
$E(q,3)$ is $3$-regular over $\mathbb{R}$. They also showed that it is
independent of ZF that $E(2,4)$ is $4$-regular over $\mathbb{R}$. We
extend these results by showing that for $n=5$ and $n=6$, it is
independent of ZF that $E(2,n)$ is $n$-regular over $\mathbb{R}$. Also,
we show that it is independent of ZF that $x_1+x_2+x_3=4x_4$ is
$4$-regular over $\mathbb{R}$.

They also show that in the Zermelo-Fraenkel-Choice (ZFC) system of
axioms, if $r$ is a positive integer and $\mathcal{L}$ is a finite
system of linear homogeneous equations with rational coefficients, then
$\mathcal{L}$ is $r$-regular over $\mathbb{R}$ if and only if
$\mathcal{L}$ is $r$-regular over $\mathbb{Z}$. It follows, in ZFC, that
the equation $E(q,3)$ is not $3$-regular over $\mathbb{R}$. In
Section~\ref{section:real}, we prove in ZFC that for every integer $n
\geq 3$, there are $2^{2^{\aleph_0}}$ different $n$-colorings of the
nonzero real numbers without a monochromatic solution to $E(q,n)$.
Hence, in ZFC, we have $\Delta(E,\mathbb{R})=2^{2^{\aleph_0}}$ for
$E=E(q,3)$ or $E=E(2,n)$ with $n \in \{3,4,5,6\}$.

Solovay \cite{So} proved that the following axiom is relatively
consistent with ZF.

\begin{axiom}[LM]
Every subset of the real numbers is Lebesgue measurable.
\end{axiom}

Notice that the axiom LM is not consistent with ZFC because, with the
axiom of choice, there are subsets of $\mathbb{R}$ that are not Lebesgue
measurable.

The following lemma is useful in proving in ZF+LM that a given linear
equation is $r$-regular over $\mathbb{R}$ for appropriate $r$.

\begin{lemma} \label{lemma:lm}
Suppose $c$ is a coloring of the nonzero real numbers that uses at most
countably many colors and there are positive real numbers $a$, $b$, $d$
such that $\log_a b$ is irrational and $c(x)=c(ax)=c(bx)\ne c(dx)$ for
all nonzero real numbers $x$.  Then there is a color class of $c$ that
is not Lebesgue measurable.
\end{lemma}

In ZF+LM, if $n\in \{3,4,5,6\}$, then for every $n$-coloring $c$ of the
nonzero real numbers without a monochromatic solution to $E(2,n)$, we
have $c(x)=c(3x)=c(5x)\ne c(2x)$ for all non-zero rational numbers $x$,
and $\log_3 5$ is irrational. Hence, by Lemma~\ref{lemma:lm}, the
equation $E(2,n)$ is $n$-regular in ZF+LM for $n \in \{3,4,5,6\}$. For
every $4$-coloring $c$ of the nonzero real numbers without a
monochromatic solution to the equation $x_1+x_2+x_3=4x_3$, we have
$c(x)=c(6x)=c(11x)\ne c(2x)$ for all nonzero rational numbers $x$ and
$\log_6 11$ is irrational. Hence, by Lemma~\ref{lemma:lm}, the equation
$x_1+x_2+x_3=4x_4$ is $4$-regular in ZF+LM.

\section{Minimal colorings over the rationals} \label{section2}
The following straightforward lemma \cite{FoKl} is useful in proving
that certain colorings are free of monochromatic solutions to particular
linear equations.

\begin{lemma} \label{lemma:valuation}
Suppose $t_1,\dotsc,t_n$ are nonzero rational numbers and $p$ is a prime
number such that $v_p(t_1) \leq v_p(t_2) \leq \dotsb \leq v_p(t_n)$ and
$\sum_{i=1}^n t_i=0$.  Then $v_p(t_1)=v_p(t_2)$.
\end{lemma}

The following proposition is due to the second author and Radoi\v{c}i\'c
~\cite{FoRa}.

\begin{proposition} \label{proposition:E2n}
For each integer $n>1$, the $n$-coloring $c_{2,n}$ is free of
monochromatic solutions to $E(2,n)$.
\end{proposition}
\begin{proof}
For $n>1$, if $x_0+2x_1+\dotsb+2^{n-2}x_{n-2}-2^{n-1}x_{n-1}=0$ is a
solution to $E(2,n)$ in the nonzero rationals, then we have by
Lemma~\ref{lemma:valuation} that for some $i$ and $j$ with $0 \leq i <j
< n$,
\[ v_2(2^{i}x_i)=v_2(2^{j}x_{j}). \]
It follows that $v_2(x_i)=(j-i)+v_2(x_j)$ and $v_2(x_j)-v_2(x_i)
\in \{1,\dotsc,n-1\}$. Therefore, $v_2(x_j)-v_2(x_i)$ is not a
multiple of $n$. Hence, $x_i$ and $x_j$ are not the same color,
and the coloring $c_{2,n}$ is free of monochromatic solutions to
$E(2,n)$.
\end{proof}

They \cite{FoRa} also prove the following structural result for
$3$-colorings of the nonzero rational numbers without a monochromatic
solution to $E(q,3)$.

\begin{lemma} \label{lemma:fr}
Let $x$ and $q$ be nonzero rational numbers with $q \ne \pm 1$, $m$ and
$n$ be integers, and let $a=\frac{q+1}{q^2}$ and $b=q(q-1)$. For every
$3$-coloring $c$ of the nonzero rational numbers without a monochromatic
solution to $E(q,3)$, we have $c(x)=c(a^mb^nx)$ if and only if $m+n$ is
a multiple of $3$.
\end{lemma}
\begin{proof}
Since $x+q(x)=q^2(ax)$, $ax$ must be a different color than $x$.  Since
$bx+q(x)=q^2(x)$, then $bx$ must be a different color than $x$. Since
$x+q(abx)=q^2(x)$, then $abx$ must be a different color than $x$. Hence,
$c(x) \ne c(rx)$ for $r \in \{a,b,ab\}$.

Recall that for a group $G$ and subset $S \subset G$, the Cayley graph
$\Gamma(G,S)$ has vertex set $G$ and two elements $x$ and $y$ are
adjacent if there is an element $s \in S$ such that $x=sy$ or $y=sx$.

We associate each rational number $a^mb^nx$ (with $m$ and $n$ integers)
with the lattice point $(m,n)$. Let $S=\{(1,0),(0,1),(1,1)\}$ and
consider the Cayley graph $\Gamma(\mathbb{Z}^2,S)$. Define the
$3$-coloring $\chi$ of $\mathbb{Z}^2$ by $\chi(m,n)=c(a^mb^nx)$. Since
$c(rx) \ne c(x)$ for $r \in \{a,b,ab\}$, then $\chi$ is a proper
$3$-coloring of the vertices of $\Gamma(\mathbb{Z}^2,S)$. By induction,
it is a straightforward check that there is only one proper coloring of
$\Gamma(\mathbb{Z}^2,S)$ up to isomorphism and this coloring is given by
$\chi(m,n) \equiv m+n \pmod{3}$. Hence, $c(x)=c(a^mb^nx)$ if and only if
$m+n$ is a multiple of $3$.
\end{proof}

\subsection{The minimal coloring for $E(2,3)$}
We prove that $c_{2,3}$ is the only minimal coloring of the nonzero
rational numbers, up to isomorphism, for the equation $E(2,3)$. By
Proposition~\ref{proposition:E2n}, we know that the $3$-coloring
$c_{2,3}$ is free of monochromatic solutions to $E(2,3)$. By
Lemma~\ref{lemma:fr}, the numbers $2$, $3$, and $4$ must be all
different colors, so $E(2,3)$ is $2$-regular. Hence $c_{2,3}$ is a
minimal coloring of the nonzero rational numbers for $E(2,3)$.

\begin{proposition} \label{proposition:E23}
The only minimal coloring of the nonzero rational numbers for $E(2,3)$
is $c_{2,3}$.
\end{proposition}
\begin{proof}
Suppose $c$ is a $3$-coloring of the nonzero rational numbers without a
monochromatic solution to $E(2,3)$. By Lemma~\ref{lemma:fr}, for every
nonzero rational number $x$ and integers $m$ and $n$, we have
$c((\frac{3}{4})^m2^nx)=c(x)$ if and only if $m+n$ is a multiple of
$3$. Equivalently, for every nonzero rational number $x$ and integers
$m$ and $n$, $c(2^m3^nx)=c(x)$ if and only if $m$ is a multiple of $3$.

For nonzero rational numbers $x$ and $r$ with $r$ positive, we now show
that $c(x)=c(rx)$ if and only if $v_2(r) \equiv 0 \pmod{3}$. By
induction, it suffices to prove that $c(x)=c(px)$ for every odd positive
integer $p$ and nonzero rational number $x$. For $p=1$ or $3$, we have
already established that $c(x)=c(px)$ for every nonzero rational number
$x$.

So let $p$ be an odd integer greater than $3$ and suppose that for every
positive odd integer $p'<p$ and rational number $x$, $c(p'x)=c(x)$. The
numbers $px$ and $4x$ are different colors because otherwise
$(x_0,x_1,x_2)=(4(p-2)x,4x,px)$ is a monochromatic solution to
$E(2,3)$. We have $px$ and $2x$ are different colors because otherwise
$(16x,2(p-4)x,px)$ is a monochromatic solution to $E(2,3)$. Since
$c(x)$, $c(2x)$, and $c(4x)$ are all different colors, then
$c(px)=c(x)$. By induction, for nonzero rational numbers $x$ and $r$
with $r$ positive, we have $c(x)=c(rx)$ if and only if $v_2(r) \equiv 0
\pmod{3}$.

To finish the proof, we need to show that $c(x)=c(-x)$ for every nonzero
rational number $x$. Since $(10x,-x,2x)$ is a solution to $E(2,3)$ and
$10x$ and $2x$ are the same color, then $-x$ and $2x$ are different
colors. Since $(12x,-8x,-x)$ is a solution to $E(2,3)$, $-8x$ and $-x$
are the same color, and $12x$ is the same color as $4x$, then $-x$ and
$4x$ are differently colored. Since $x$, $2x$, and $4x$ are all
different colors, then $c(-x)=c(x)$. Therefore, $c_{2,3}$ is the only
minimal coloring, up to isomorphism, of the nonzero rational numbers
without a monochromatic solution to $E(2,3)$.
\end{proof}

Using a very similar argument to the proof of
Proposition~\ref{proposition:E23}, it is not difficult to show that
there are only two minimal colorings of the positive integers for the
equation $E(2,3)$. These two minimal colorings consist of $c_{2,3}$ with
its domain restricted to the positive integers and a coloring $c_{2,3}'$
that is identical to $c_{2,3}$ with its domain restricted to the
positive integers except that the color of $1$ is different. Formally,
$c_{2,3}'\colon \mathbb{N} \rightarrow \{0,1,2\}$ is the coloring of the
positive integers such that $c_{2,3}'(1)=2$ and $c_{2,3}'(n)=c_{2,3}(n)$
for $n>1$.

\begin{proposition}
The only two minimal colorings of the positive integers for the equation
$E(2,3)$ are the colorings $c_{2,3}$ with its domain restricted to the
positive integers and $c_{2,3}'$.
\end{proposition}

\subsection{The minimal coloring for $E(2,4)$}

In this subsection we prove that the only minimal coloring of the
nonzero rational numbers, up to isomorphism, for $E(2,4)$ is
$c_{2,4}$. The proof uses the following lemma from \cite{FoRa}, proven
below.

\begin{lemma} \label{lemma:muliple4}
For every $4$-coloring $c$ of the nonzero rational numbers without a
monochromatic solution to $E(2,4)$ and for nonzero rational number $x$
and integers $m$ and $n$, we have $c(x)=c(2^m3^nx)$ if and only if $m$
is a multiple of $4$.
\end{lemma}

In particular, Lemma~\ref{lemma:muliple4} implies that $E(2,4)$ is
$3$-regular since in every $4$-coloring of the nonzero rational numbers
without a monochromatic solution to $E(2,4)$, the numbers $1$, $2$, $4$,
and $8$ are different colors. It follows that $c_{2,4}$ is minimal for
$E(2,4)$.  Moreover, we have the following proposition.

\begin{proposition} \label{proposition:E24}
The coloring $c_{2,4}$ is the only minimal coloring of the nonzero
rational numbers without a monochromatic solution to $E(2,4)$.
\end{proposition}
\begin{proof}[Proof of Lemma~\ref{lemma:muliple4}.]
By induction on $m$ and $n$, it suffices to prove that for all nonzero
rational numbers $q$, we have $c(q)=c(3q)=c(16q)$ and $c(q) \notin
\{c(2q),c(4q),c(8q)\}$.  By considering all solutions to $E(2,4)$ with
exactly two distinct variables, for $r \in \{\frac{n+1}{n} : n \in
\mathbb{Z}~\textrm{and}~ 1 \leq n \leq 7\}$, we have that $q$ and $rq$
are different colors (call these ratios $r$ \emph{forbidden}). Note
immediately that $c(x) \ne c(2x)$ by the forbidden ratio $2$.

We begin by showing that $c(x) \ne c(4x)$.  Proceeding by contradiction,
suppose $c(x) = c(4x)$ instead for some $x$.  By forbidden ratios
amongst themselves, we have that $c(x) = c(4x)$, $c(2x)$, $c(3x)$, and
$c(\frac{3}{2}x)$ are distinct colors.  Then $c(\frac{9}{4}x) = c(2x)$
by forbidden ratios from $3x$ and $\frac{3}{2}x$ and because of the
solution $(x,4x,\frac{9}{4}x,\frac{9}{4}x)$.  Similarly, $c(6x) =
c(\frac{3}{2}x)$ because of forbidden ratios from $4x$ and $3x$ and
because of the solution $(6x,2x,2x,\frac{9}{4}x)$.  Finally,
$\frac{18}{7}x$ cannot be colored with any colors because of forbidden
ratios from $\frac{9}{4}x$ and $3x$ as well as the solutions
$(\frac{18}{7}x, x, 4x, \frac{18}{7}x)$ and $(\frac{18}{7}x, 6x,
\frac{3}{2}x, \frac{18}{7}x)$.

We now show that $c(x) = c(3x)$.  Assume otherwise, so that $c(x) \ne
c(3x)$ for some $x$.  Once again, by forbidden ratios amongst
themselves, we have that $c(x)$, $c(3x)$, $c(2x)$, and $c(4x)$ are
distinct; furthermore, $c(\frac{3}{2}x) = c(4x)$ also by forbidden
ratios.  We have $c(\frac{4}{3}x) = c(3x)$ by forbidden ratios from $x$
and $2x$ as well as the solution $(4x, \frac{4}{3}x, \frac{4}{3}x,
\frac{3}{2}x)$.  Next, $c(\frac{5}{3}x) = c(x)$ by forbidden ratios from
$\frac{4}{3}x$ and $2x$ as well as the solution $(4x, \frac{5}{3}x,
\frac{3}{2}x, \frac{5}{3}x)$.  Similarly, $c(6x) = c(2x)$ by forbidden
ratios from $3x$ and $\frac{3}{2}x$ and the solution $(6x, \frac{5}{3}x,
x, \frac{5}{3}x)$.  Finally, $\frac{9}{4}x$ cannot be colored with any
colors, because of forbidden ratios from $3x$ and $\frac{3}{2}x$ as well
as the solutions $(x, \frac{5}{3}x, \frac{9}{4}x, \frac{5}{3}x)$ and
$(6x, 2x, 2x, \frac{9}{4}x)$.

Clearly, $c(x) \ne c(8x)$ since otherwise $(8x, \frac{8}{3}x,
\frac{8}{3}x, 3x)$ would be a monochromatic solution.  Completing the
proof, both $c(x)$ and $c(16x)$ must be different from all of $c(2x)$,
$c(4x)$, and $c(8x)$ (which are distinct), so $c(x) = c(16x)$.
\end{proof}

\begin{proof}[Proof of Proposition~\ref{proposition:E24}.]
The proof is similar to the proof of
Proposition~\ref{proposition:E23}. Suppose $c$ is a $4$-coloring of the
nonzero rational numbers without a monochromatic solution to
$E(2,4)$. By Lemma~\ref{lemma:muliple4}, for every nonzero rational
number $x$ and integers $m$ and $n$, we have $c(x)=c(2^m3^nx)$ if and
only if $m$ is a multiple of $4$. We now show that for every positive
rational number $r$ with $v_2(r) \equiv 0 \pmod{4}$, we have
$c(x)=c(rx)$. By induction, it suffices to prove that $c(x)=c(px)$ for
every positive odd integer $p$ and nonzero rational number $x$. For
$p=1$ or $3$, Lemma~\ref{lemma:muliple4} implies that $c(x)=c(px)$ for
every nonzero rational number $x$.

For $p=5$, we have $(18x,5x,5x,6x)$ is a solution to $E(2,4)$ and
$c(18x)=c(2x)=c(6x)$, so $5x$ and $2x$ are different colors.  Also,
$(12x,4x,5x,5x)$ is a solution to $E(2,4)$ and $c(12x)=c(4x)$, so $5x$
and $4x$ are different colors.  Finally, $(5x,\frac{3}{2}x,8x,5x)$ is a
solution to $E(2,4)$ and $c(\frac{3}{2}x)=c(8x)$, so $c(5x)$ and $c(8x)$
are different colors. Since $x$, $2x$, $4x$, and $8x$ are all different
colors, then $c(5x)$ must be the color of $c(x)$.

The rest of the proof is by induction on $p$.  Suppose $p$ is an odd
integer greater than $5$ and that for all positive odd $p' < p$ and
rational $x$, $c(p'x) = c(x)$.  Then for all $x$, we know that $px$ and
$2x$ are different colors since $(\frac{32}{3}x, \frac{32}{3}x,
2(p-4)x,px)$ would otherwise be a monochromatic solution. Similarly,
$px$ and $4x$ are different colors because of $(\frac{64}{5}x, 4(p-2)x,
\frac{4}{5}x, px)$.  Finally, $px$ and $8x$ are different colors because
of $(8(p-2)x, \frac{8}{3}x, \frac{8}{3}x, px)$. Since $c(x)$, $c(2x)$,
$c(4x)$, and $c(8x)$ are all distinct, it follows that $c(px) = c(x)$ as
desired.

To finish the proof, we need to show that $c(x)=c(-x)$ for every nonzero
rational number $x$. Since $(10x,-x,2x,2x)$ is a solution to $E(2,4)$
and $c(10x)=c(2x)$, then $c(-x)$ and $c(2x)$ are different colors. Since
$(28x,-16x,-x,-x)$ is a solution to $E(2,4)$, $c(-16x)=c(-x)$, and
$c(28x)=c(4x)$, then $-x$ and $4x$ are different colors. Since
$(88x,-16x,-16x,-x)$ is a solution to $E(2,4)$, $c(-16x)=c(-x)$, and
$c(88x)=c(8x)$, then $-x$ and $8x$ are different colors. Since $x$,
$2x$, $4x$, and $8x$ are all different colors, then
$c(-x)=c(x)$. Therefore, $c_{2,4}$ is the only minimal coloring, up to
isomorphism, of the nonzero rational numbers without a monochromatic
solution to $E(2,4)$.
\end{proof}

\subsection{The minimal colorings for $E(\frac{3}{2},3)$}

In this subsection we prove that the two minimal colorings of the
nonzero rational numbers for $E(\frac{3}{2},3)$ are $c_{2,3}$ and
$c_{3,3}$. By a proof similar to Proposition~\ref{proposition:E2n}, it
is clear that $c_{2,3}$ and $c_{3,3}$ are free of monochromatic
solutions to $E(\frac{3}{2},3)$. The equation $E(\frac{3}{2},3)$ is
$2$-regular since in any coloring of the nonzero rational numbers
without a monochromatic solution to $E(\frac{3}{2},3)$, the numbers $9$,
$10$, and $12$ are all different colors (by Lemma~\ref{lemma:fr}). Hence
$c_{2,3}$ and $c_{3,3}$ are minimal colorings for
$E(\frac{3}{2},3)$. Again from Lemma~\ref{lemma:fr}, it follows that for
every nonzero rational number $x$ and integers $m$ and $n$, we have
$c(x)=c((\frac{6}{5})^m(\frac{10}{9})^nx)$ if and only if $m+n$ is a
multiple of $3$. In particular, $c(x)=c(rx)$ for $r \in \{\frac{8}{5},
\frac{64}{27}, \frac{40}{27}\}$ and $c(x) \ne c(rx)$ for $r \in
\{\frac{6}{5}, \frac{10}{9},\frac{4}{3}\}$.

For the rest of this subsection, we suppose that $c$ is a $3$-coloring
that is free of monochromatic solutions to $E(\frac{3}{2},3)$ and we
will deduce that $c$ is either $c_{2,3}$ or $c_{3,3}$. We first build up
structural properties about $c$ if there is a nonzero rational number
$x$ such that $c(x)=c(2x)$ and deduce that $c$ is the coloring
$c_{3,3}$. We then prove properties about $c$ if $c(x) \ne c(2x)$ for
every nonzero rational number $x$.  We deduce in this case that $c$ is
the coloring $c_{2,3}$.

\begin{lemma} \label{lemma:2powers}
If $x$ is a nonzero rational number such that $c(2x)=c(x)$, then
$c(2^nx)=c(x)$ for every nonnegative integer $n$.
\end{lemma}

\begin{lemma} \label{lemma:3multiples}
If $x$ is a nonzero rational number such that $c(2x)=c(x)$, then for all
nonnegative integers $k$, $m$, and $n$ we have $c(2^k3^m5^nx)=c(x)$ if
and only if $m$ is a multiple of $3$.
\end{lemma}

\begin{lemma} \label{lemma:36}
If $x$ is a nonzero rational number such that $c(2x)=c(x)$, then
$c(3x)=c(6x)$.
\end{lemma}

From Lemma~\ref{lemma:3multiples} and Lemma~\ref{lemma:36}, we have the
following result.

\begin{lemma}  \label{lemma:235}
If $c$ is a $3$-coloring of the nonzero rational numbers without a
monochromatic solution to $E(\frac{3}{2},3)$ and $x$ is a rational
number such that $c(2x)=c(x)$, then for all integers $m_1$, $n_1$,
$p_1$, $m_2$, $n_2$, $p_2$, we have
$c(2^{n_1}3^{m_1}5^{p_1}x)=c(2^{n_2}3^{m_2}5^{p_2}x)$ if and only if
$m_1 \equiv m_2 \pmod{3}$.
\end{lemma}

The following lemma gets us much closer to proving that $c_{3,3}$ is the
only $3$-coloring of the nonzero rational numbers for which there is a
nonzero rational number $x$ such that $c(x)=c(2x)$.

\begin{lemma} \label{lemma:3multiples2}
If $x$ is a nonzero rational number such that $c(x)=c(2x)$, then for
every positive integer $n$, we have $c(nx)=c(x)$ if and only if $v_3(n)$
is a multiple of $3$.
\end{lemma}

Finally, to finish the proof that $c_{3,3}$ is the only $3$-coloring $c$
of the nonzero rational numbers without a monochromatic solution to
$E(\frac{3}{2},3)$ and for which there is a rational number $x$ such
that $c(x)=c(2x)$, it suffices to prove that $c(y)=c(-y)$ for all $y$.

\begin{lemma} \label{lemma:onlyc33}
The only $3$-coloring $c$ of the nonzero rational numbers for which
there is a rational number $x$ such that $c(x)=c(2x)$ is $c_{3,3}$.
\end{lemma}

Having completed the case when there is a nonzero rational number $x$
such that $c(x)=c(2x)$, we now look at those $3$-colorings for which
$c(x) \ne c(2x)$ for all nonzero rational numbers $x$.

\begin{lemma} \label{lemma:iff3multiple}
If $c(x) \ne c(2x)$ for every nonzero rational number $x$, then
$c(2^ny)=c(y)$ holds for integer $n$ and nonzero rational number $y$ if
and only if $n$ is a multiple of $3$.
\end{lemma}

\begin{lemma} \label{lemma:iff3multiple2}
If $c(x) \ne c(2x)$ for every nonzero rational number $x$, then
$c(2^m3^n5^py)=c(y)$ holds for nonzero rational number $y$ and integers
$m$, $n$, and $p$ if and only if $m$ is a multiple of $3$.
\end{lemma}

\begin{lemma} \label{lemma:neg}
If $c(x) \ne c(2x)$ for all nonzero rational numbers $x$, then
$c(y)=c(-y)$ for all nonzero rational numbers $y$.
\end{lemma}

Finally, to finish the proof of Proposition~\ref{proposition7} that
$c_{2,3}$ and $c_{3,3}$ are the only $3$-colorings of the nonzero
rational numbers without a monochromatic solution to $E(\frac{3}{2},3)$,
it suffices to prove the following lemma.

\begin{lemma} \label{lemma:3multiple3}
If no nonzero rational number $x$ satisfies $c(x)=c(2x)$, then
$c(nx)=c(x)$ if and only if $v_2(n)$ is a multiple of $3$.
\end{lemma}

\begin{proof}[Proof of Lemma~\ref{lemma:2powers}.]
By induction on $n$, it suffices to prove that $c(4x)=c(x)$ if
$c(x)=c(2x)$.  Assume for contradiction that $c(x) = c(2x) \ne c(4x)$.
Then we know that $c(3x)$ must further be distinct from $c(x)$ and
$c(4x)$ because of the forbidden ratio from $4x$ and the solution
$(3x,x,2x)$.  It follows that $c(\frac{10}{3}x) = c(x)$ because of
forbidden ratios from $3x$ and $4x$.  Similarly, $c(\frac{5}{2}x) =
c(4x)$ because of forbidden ratios from $3x$ and $\frac{10}{3}x$.  It
follows that $c(\frac{13}{3}x) = c(3x)$ because of the solutions $(x,
\frac{13}{3}x, \frac{10}{3}x)$ and $(\frac{5}{2}x, \frac{13}{3}x, 4x)$.
Similarly, $c(\frac{9}{2}x) = c(4x)$ because of the solutions
$(\frac{9}{2}x, 2x, \frac{10}{3}x)$ and $(3x, \frac{9}{2}x,
\frac{13}{3}x)$.  We have $c(6x) = c(3x)$ because of a forbidden ratio
from $\frac{9}{2}x$ and the solution $(6x, x, \frac{10}{3}x)$.  Finally,
the number $\frac{3}{4}x$ cannot be colored with any colors, because of
a forbidden ratio from $x$ as well as the solutions $(\frac{9}{2}x,
\frac{3}{4}x, \frac{5}{2}x)$ and $(\frac{3}{4}x, 6x, \frac{13}{3}x)$.
\end{proof}

\begin{proof}[Proof of Lemma~\ref{lemma:3multiples}.]
By the previous lemma, we have $c(2^nx)=c(x)$ for all nonnegative
integers $n$. Since $c(\frac{5}{8}y)=c(y)$ for every nonzero rational
number $y$, then $c(2^n5^px)=c(x)$ for all nonnegative integers $n$ and
$p$. To finish the proof, it suffices, by induction, to prove that
neither $3x$ nor $9x$ is the same color as $x$, and $27x$ is the same
color as $x$. Since $3x$ and $4x$ have ratio $\frac{3}{4}$, then $c(3x)
\ne c(4x)=c(x)$. Since $9x$ and $16x$ have ratio
$(\frac{6}{5})^{-2}(\frac{10}{9})^{-2}$, then $c(9x) \ne
c(16x)=c(x)$. Since $27x$ and $64x$ have ratio $\frac{27}{64}$, then
$c(27x)=c(64x)=c(x)$.
\end{proof}

\begin{proof}[Proof of Lemma~\ref{lemma:36}.]
Assume for contradiction that $c(3x) \ne c(6x)$. By
Lemma~\ref{lemma:3multiples}, $c(x)=c(4x)=c(8x)$. Since
$4x=\frac{4}{3}(3x)$ and $8x=\frac{4}{3}(6x)$, then $x$, $3x$, and $6x$
are all different colors. Since $\frac{64}{27}(6x)=\frac{32}{9}x$, then
$\frac{32}{9}x$ is the same color as $6x$. Since
$(6x,\frac{4}{3}x,\frac{32}{9}x)$ is a solution to $E(\frac{3}{2},3)$,
then $c(\frac{4}{3}x)$ and $c(6x)$ are different colors. Since
$\frac{4}{3}x$ and $x$ have ratio $\frac{4}{3}$, then $\frac{4}{3}x$ and
$x$ are different colors. Hence, $\frac{4}{3}x$ is the same color as
$3x$. Since $\frac{3}{4}x$, $x$, and $\frac{4}{3}x$ are all different
colors, then $\frac{3}{4}x$ is the same color as $6x$. Since
$\frac{3}{4}(2x)=\frac{3}{2}x$, then $\frac{3}{2}x$ and $x$ are
different colors. If $\frac{3}{2}x$ and $3x$ are the same color, then by
Lemma~\ref{lemma:3multiples}, $6x$ would also be the same color as $3x$,
contradicting the assumption that $3x$ and $6x$ are different
colors. Hence, $\frac{3}{2}x$ is the same color as $6x$. But then
$\frac{3}{2}x$ and $\frac{3}{4}x$ are the same color, which implies, by
Lemma~\ref{lemma:3multiples}, that $3x$ is the same color as $6x$, a
contradiction.
\end{proof}

\begin{proof}[Proof of Lemma~\ref{lemma:235}.]
Clearly, from Lemma~\ref{lemma:3multiples} and Lemma~\ref{lemma:36}, we
have the result for \emph{nonnegative} integers.

Recall that, for every nonzero rational number $y$, we have
$c(y)=c(\frac{8}{5}y)$ and we also have $y$, $\frac{4}{3}y$, and
$\frac{16}{9}y$ are all different colors. Therefore, it suffices, by the
remark above and induction to prove that $c(\frac{x}{2})=c(x)$. Since
$c(3x)=c(6x)$ and $(6x,\frac{x}{2},3x)$ is a solution to
$E(\frac{3}{2},3)$, then $\frac{x}{2}$ is a different color from
$3x$. We have $\frac{80}{3}x=\frac{40}{27}(18x)$, so $18x$ and
$\frac{80}{3}x$ are the same color as $9x$. Since
$(\frac{x}{2},\frac{80}{3}x,18x)$ is a solution to $E(\frac{3}{2},3)$,
then $\frac{x}{2}$ and $9x$ are different colors. Hence, $\frac{x}{2}$
is the same color as $x$, which completes the proof.
\end{proof}

\begin{proof}[Proof of Lemma~\ref{lemma:3multiples2}.]
Suppose $c$ is a $3$-coloring of the nonzero rational numbers without a
monochromatic solution to $E(\frac{3}{2},3)$ and $x$ is a nonzero
rational number such that $c(x)=c(2x)$. By Lemma~\ref{lemma:235}, for
integers $m_1$, $n_1$, $p_1$, $m_2$, $n_2$, and $p_2$ we have
$c(2^{n_1}3^{m_1}5^{p_1}x)=c(2^{n_2}3^{m_2}5^{p_2}x)$ if and only if
$m_2-m_1$ is a multiple of $3$. By induction, it suffices to prove for
every prime $p>3$, that $c(x)=c(px)$. For $p=5$, Lemma~\ref{lemma:235}
implies that $c(x)=c(px)$.

The proof is by induction on the size of $p$. Suppose $p$ is prime with
$p>5$. We write $p=6a+b$, where $a$ and $b$ are nonnegative integers and
$b \in \{1,5\}$. The induction hypothesis is that $c(q)=c(p'q)$ for
every nonzero rational number $q$ satisfying $c(q)=c(2q)$ and prime $p'$
satisfying $3<p'<p$. The induction hypothesis implies that $c(q)=c(p'q)$
for every nonzero rational number $q$ satisfying $c(q)=c(2q)$ and
positive odd integer $p'$ which is not a multiple of $3$ and satisfies
$p'<p$. Then we have for all $p$ that $c(px) \ne c(9x) =
c(\frac{9}{4}x)$ because of the solution $(\frac{9}{4}(p-6)x, 9x, px)$.
It suffices to show that in the two cases below, $c(px) \ne c(3x)$ since
$c(x)$, $c(3x)$, and $c(9x)$ are distinct.

\textbf{Case 1:} $p=6a+1$. For $a=1$, we have $p=7$, and $(3x,7x,6x)$ is
a solution to $E(\frac{3}{2},3)$, so $7x$ and $3x$ are different
colors. For $a>1$, we have $0<3a+5<6a+1$, so by the induction
hypothesis, we have $c((3a+5)q)=c(q)$ for every rational number $q$
satisfying $c(q)=c(2q)$, and in particular, for $q=\frac{8}{9}x$. The
numbers $6x$ and $\frac{8}{9}x$ are the same color as the color of $3x$
by Lemma~\ref{lemma:235}. Since $(px,6x,(3a+5)\frac{8}{9}x)$ is a
solution to $E(\frac{3}{2},3)$, then $px$ and $3x$ are different
colors. Hence, $c(px)=c(x)$.

\textbf{Case 2:} $p=6a+5$. For each prime $p>5$ of the form $p=6a+5$ we
have $3a+7<6a+5$, so by the induction hypothesis we have
$c((3a+5)q)=c(q)$ for every rational number $q$ satisfying $c(q)=c(2q)$,
and in particular, for $q=\frac{8}{9}x$. The numbers $6x$ and
$\frac{8}{9}x$ are the same color as the color of $3x$ by
Lemma~\ref{lemma:235}. Since $(px,6x,(3a+7)\frac{8}{9}x)$ is a solution
to $E(\frac{3}{2},3)$, then $px$ and $3x$ are different colors. Hence,
$c(px)=c(x)$.
\end{proof}

\begin{proof}[Proof of Lemma~\ref{lemma:onlyc33}.]
By Lemma~\ref{lemma:3multiples2}, it suffices to prove that $c(-y)=c(y)$
for all nonzero rational numbers $y$. By Lemma~\ref{lemma:3multiples2},
we have $c(\frac{85}{6}y)=c(9y)$. Since $(-y,\frac{85}{6}y,9y)$ is a
solution to $E(\frac{3}{2},3)$, then $-y$ and $9y$ are different
colors. By Lemma~\ref{lemma:3multiples2}, we have
$c(\frac{14}{9}y)=c(3y)$. Since $(-y,3y,\frac{14}{9}y)$ is a solution to
$E(\frac{3}{2},3)$, then $-y$ and $3y$ are different colors. Therefore,
we have $c(-y)=c(y)$, completing the proof.
\end{proof}

\begin{proof}[Proof of Lemma~\ref{lemma:iff3multiple}.]
It suffices to prove that $c(y) \ne c(4y)$ for all nonzero rational
numbers $y$. So suppose for contradiction that there is a nonzero
rational number $y$ such that $c(y)=c(4y)$. Let red be the color of $y$,
blue be the color of $2y$, and green be the remaining color.  Since
$3y=\frac{4}{3}(4y)$, then $3y$ is green or blue.

\textbf{Case 1:} $3y$ is green. Since $3y$, $4y$, and $\frac{16}{3}y$
are all different colors, then $\frac{16}{3}$y is blue. Since
$\frac{9}{4}y=\frac{27}{64}(\frac{16}{3}y)$, then $\frac{9}{4}y$ is
blue. Since $\frac{9}{2}y=2(\frac{9}{4}y)$, then $\frac{9}{2}y$ is not
blue. Since $(9y,2y,\frac{16}{3}y)$ is a solution to $E(\frac{3}{2},3)$,
then $9y$ is not blue. Since $9y=2(\frac{9}{2}y)$, then $\frac{9}{2}y$
and $9y$ are different colors. Therefore, either $\frac{9}{2}y$ is red
and $9y$ is green or $\frac{9}{2}y$ is green and $9y$ is red.

\textbf{Case 1a:} $\frac{9}{2}y$ is red and $9y$ is green. Since
$6y=\frac{4}{3}(\frac{9}{2}y)$, then $6y$ is not red. Since $6y=2(3y)$,
then $6y$ is not green. Hence, $6y$ is blue. Since $\frac{9}{2}y$, $6y$,
and $8y$ are all different colors, then $8y$ is green. Since
$12y=2(6y)$, then $12y$ is not blue. Since $12y=\frac{4}{3}(9y)$, then
$12y$ is red. Since $\frac{15}{2}y=\frac{5}{8}(12y)$, then
$\frac{15}{2}y$ is red. Then $(\frac{15}{2}y,y,4y)$ is a monochromatic
solution to $E(\frac{3}{2},3)$, a contradiction.

\textbf{Case 1b:} $\frac{9}{2}y$ is green and $9y$ is red. Since
$\frac{10}{3}y=\frac{5}{8}(\frac{16}{3}y)$, then $\frac{10}{3}y$ is
blue. Since $3y=2(\frac{3}{2}y)$ and $2y=\frac{4}{3}(\frac{3}{2}y)$,
then $\frac{3}{2}y$ is red. Since
$\frac{5}{3}y=\frac{10}{9}(\frac{3}{2}y)$ and
$\frac{5}{3}y=\frac{5}{6}(2y)$, then $\frac{5}{3}y$ is green.  Finally,
$\frac{28}{9}y$ can not be colored with any colors because of the
solutions $(y,4y,\frac{28}{9}y)$, $(2y,\frac{10}{3}y,\frac{28}{9}y)$,
and $(\frac{9}{2}y,\frac{5}{3}y,\frac{28}{9}y)$.

\textbf{Case 2:} $3y$ is blue. Since $3y$, $4y$, $\frac{16}{3}y$ are all
different colors, then $\frac{16}{3}y$ is green. Since
$\frac{16}{3}y=2(\frac{8}{3}y)$, then $\frac{8}{3}y$ is not green.
Since $\frac{8}{3}y=\frac{4}{3}(2y)$, then $\frac{8}{3}y$ is not
blue. Hence, $\frac{8}{3}y$ is red. Since $\frac{3}{2}y$, $2y$,
$\frac{8}{3}y$ are all different colors, then $\frac{3}{2}y$ is
green. Since $\frac{9}{4}=\frac{27}{64}(\frac{16}{3}y)$, then
$\frac{9}{4}y$ is green. Since $\frac{9}{2}y=2(\frac{9}{4}y)$, then
$\frac{9}{2}y$ is not green. Since $(\frac{9}{2}y,y,\frac{8}{3}y)$ is a
solution to $E(\frac{3}{2},3)$, then $\frac{9}{2}y$ is not
red. Therefore, $\frac{9}{2}y$ is blue. Since
$8y=(\frac{6}{5})^2(\frac{10}{9})^2(\frac{9}{2}y)$, then $8y$ is not
blue. Since $8y=2(4y)$, then $8y$ is not red. Hence, $8y$ is
green. Since $5y=\frac{5}{8}(8y)$, then $5y$ is green.  Since
$\frac{32}{9}y=\frac{64}{27}(\frac{3}{2}y)$, then $\frac{32}{9}$ is
green. Since $(\frac{y}{2},5y,\frac{32}{9}y)$ is a solution to
$E(\frac{3}{2},3)$, then $\frac{y}{2}$ is not green. Since
$y=2(\frac{y}{2})$, then $\frac{y}{2}$ is not red. Hence, $\frac{y}{2}$
is blue.

We found a contradiction in Case 1, so if $y$ and $4y$ are the same
color, then $\frac{y}{2}$, $2y$, $3y$ are the same color.  Therefore,
$\frac{y}{4}$, $y$, $\frac{3}{2}y$ must be the same color. But
$\frac{3}{2}y$ is green and $y$ is red, a contradiction.
\end{proof}

\begin{proof}[Proof of Lemma~\ref{lemma:iff3multiple2}.]
It suffices, by induction and Lemma~\ref{lemma:iff3multiple}, to prove
that $c(y)=c(3y)=c(5y)$ for every $y$. By
Lemma~\ref{lemma:iff3multiple}, $c(y)=c(8y)$ for every $y$. Since
$8y=\frac{8}{5}(5y)$, then $5y$ is the same color as $8y$, so
$c(y)=c(5y)$ for all $y$. So suppose for contradiction that there is a
nonzero rational number $x$ such that $c(x) \ne c(3x)$. So, by
Lemma~\ref{lemma:iff3multiple}, $x$, $2x$, and $4x$ are all different
colors. Let red be the color of $x$, blue be the color of $2x$, and
green be the color of $4x$. Since $3x$ is a different color from $x$ and
$4x$, then $3x$ is blue. Since $24x=8(3x)$, then $24x$ is blue. Since
$50x=25(2x)$, then $50x$ is blue. Since
$\frac{64}{9}x=\frac{64}{27}(3x)$, then $\frac{64}{9}x$ is blue. Since
$(3x,24x,\frac{52}{3}x)$, $(3x,50x,\frac{104}{3}x)$, and
$(3x,\frac{26}{3}x,\frac{64}{9}x)$ are solutions to $E(\frac{3}{2},3)$,
then none of the numbers $\frac{26}{3}x$, $\frac{52}{3}x$,
$\frac{104}{3}x$ is blue. But $\frac{26}{3}x$, $\frac{52}{3}x$,
$\frac{104}{3}x$ are all different colors, so one of them has to be
blue, a contradiction.
\end{proof}

\begin{proof}[Proof of Lemma~\ref{lemma:neg}.]
By Lemma~\ref{lemma:iff3multiple2}, the numbers $2y$ and $\frac{2}{9}y$
are the same color and the numbers $4y$, $\frac{4}{3}y$, $\frac{4}{9}y$
are the same color. Since $(2y,-y,\frac{2}{9}y)$ and
$(-y,\frac{4}{3}y,\frac{4}{9}y)$ are solutions to $E(\frac{3}{2},3)$ and
$y$, $2y$, $4y$ are all different colors, then $c(y)=c(-y)$.
\end{proof}

\begin{proof}[Proof of Lemma~\ref{lemma:3multiple3}.]
By Lemma~\ref{lemma:iff3multiple2} and Lemma~\ref{lemma:neg}, for
integers $l$, $m$, $n$, and $p$ and nonzero rational number $y$, we have
$c((-1)^{l}2^{m}3^{n}5^{p}y)=c(y)$ if and only if $m$ is a multiple of
$3$. By induction, it suffices to prove for every odd $p \geq 3$, that
$c(x)=c(px)$. For $p=3$ or $p=5$, Lemma~\ref{lemma:iff3multiple2}
implies that $c(x)=c(px)$.

The proof is by induction on the size of $p$. Suppose $p$ is an odd
number with $p>5$. The induction hypothesis is that $c(q)=c(p'q)$ for
every nonzero rational number $q$ and odd number $p'$ that is less than
$p$.  We know that $c(px) \ne c(2x) = c(\frac{9}{4}x)=c(6x)$ because of
the solution $(\frac{9}{4}(p-4)x, 6x, px)$.  Similarly, $c(px) \ne c(4x)
= c(\frac{3}{2}x)=c(\frac{9}{2}x)$ because of the solution
$(\frac{9}{2}x, \frac{3}{2}(p-2)x, px)$.  Since $c(x)$, $c(2x)$, and
$c(4x)$ are distinct, $c(px) = c(x)$.
\end{proof}
\section{Minimal colorings for $x_1+x_2+x_3=4x_4$} \label{section3}
In this section we prove that the minimal colorings for
$x_1+x_2+x_3=4x_4$ are those of the form $c_{\Pi_5}$. It is a
straightforward check that each of the colorings of the form $c_{\Pi_5}$
is minimal for $x_1+x_2+x_3=4x_4$. We suppose for the rest of this
subsection that $c$ is a $4$-coloring of the nonzero rational numbers
without a monochromatic solution to $x_1+x_2+x_3=4x_4$.

\begin{lemma} \label{lemma:forb1}
If $x$ is a nonzero rational number and $r \in
\{\frac{4}{3},\frac{3}{2},2\}$, then $c(x) \ne c(rx)$.
\end{lemma}

\begin{lemma} \label{lemma:forb2}
For every nonzero rational number $x$, we have $c(x) \ne c(3x)$.
\end{lemma}

\begin{lemma} \label{lemma:forb3}
For every nonzero rational number $x$, we have $c(x) \ne c(4x)$.
\end{lemma}

\begin{lemma} \label{lemma:235-2}
For every nonzero rational number $x$ and integers $m$ and $n$, we have
$c(x) = c(2^m3^nx)$ if and only if $w_5(2^m3^n) \equiv 1 \pmod{5}$.
\end{lemma}

\begin{lemma} \label{lemma:-4}
For every nonzero rational number $x$, we have $c(-x)=c(4x)$.
\end{lemma}

To finish the proof of Proposition~\ref{proposition10}, it suffices by
Lemma~\ref{lemma:-4} to prove the following lemma.

\begin{lemma} \label{lemma:mod5}
If $x$ is a nonzero rational number and $n$ is a positive integer
satisfying $v_5(n)=0$ and $n \equiv d \pmod{5}$ with $d\in \{1,2,3,4\}$,
then $c(nx)=c(dx)$.
\end{lemma}

\begin{proof}[Proof of Lemma~\ref{lemma:forb1}.]
Since $(\frac{4}{3}x,\frac{4}{3}x,\frac{4}{3}x,x)$,
$(\frac{3}{2}x,\frac{3}{2}x,x,x)$, and $(2x,x,x,x)$ are solutions to
$x_1+x_2+x_3=4x_4$, then $c(x)\ne c(rx)$ for $r \in
\{\frac{4}{3},\frac{3}{2},2\}$.
\end{proof}

\begin{proof}[Proof of Lemma~\ref{lemma:forb2}.]
  We suppose for contradiction that there is a nonzero rational number
  $x$ such that $c(x)=c(3x)$. Without loss of generality, we may take
  $x=1$. The previous lemmas imply that the ratios $2$, $\frac 32$, and
  $\frac 43$ are \emph{forbidden ratios}, that is, $c(x) \ne c(rx)$ for
  $r\in \{2, \frac 32, \frac 43\}$. Using these forbidden ratios, we see
  that $c(1)=c(3)$, $c(2)$, and $c(4)$ must be different colors.

  \newcommand{\siz}{\footnotesize }
  \newcommand{\tab}{\hspace{1em}}
  \newcommand{\myfrac}[2]{\ensuremath{\frac{#1}{#2}}}
  \renewcommand{\arraystretch}{0.7}
  \begin{table}
    \begin{tabular}{|l|r@{}c@{}l|l|l|l|l|} \hline \siz
      Assumptions &\multicolumn{3}{|l|}{\siz Claim} &\siz  Why not 0 &\siz  Why not 1 &\siz  Why not 2 &\siz  Why not 3 \\\hline \siz
      $c(1)=c(3)=0, c(2) = 1, c(4) = 2$ &\siz $c(6)$&\siz ${}\in{}$&\siz $\{1,3\}$&\siz $3\cdot2 = 6$&\siz &\siz $4\cdot\myfrac{3}{2} = 6$&\siz \\ \hline \siz
      $c(6) = 1$&\siz $c(8)$&\siz ${}={}$&\siz $3$&\siz $(3,1,8,3)$&\siz $6\cdot\myfrac{4}{3} = 8$&\siz $4\cdot2 = 8$&\siz \\ \hline \siz
      \tab $c(8) = 3$&\siz $c(12)$&\siz ${}\in{}$&\siz $\{0,2\}$&\siz &\siz $6\cdot2 = 12$&\siz &\siz $8\cdot\myfrac{3}{2} = 12$\\ \hline \siz
      \tab\tab $c(12) = 0$&\siz $c(16)$&\siz ${}={}$&\siz $2$&\siz $12\cdot\myfrac{4}{3} = 16$&\siz $(6,2,16,6)$&\siz &\siz $8\cdot2 = 16$\\ \hline \siz
      \tab\tab\tab $c(16) = 2$&\siz $c(9)$&\siz ${}={}$&\siz $3$&\siz $12\cdot\myfrac{3}{4} = 9$&\siz $6\cdot\myfrac{3}{2} = 9$&\siz $(16,16,4,9)$&\siz \\ \hline \siz
      \tab\tab\tab\tab $c(9) = 3$&\siz $c(\myfrac{9}{2})$&\siz ${}={}$&\siz $2$&\siz $3\cdot\myfrac{3}{2} = \myfrac{9}{2}$&\siz $6\cdot\myfrac{3}{4} = \myfrac{9}{2}$&\siz &\siz $9\cdot\myfrac{1}{2} = \myfrac{9}{2}$\\ \hline \siz
      \tab\tab\tab\tab\tab $c(\myfrac{9}{2}) = 2$&\siz $c(10)$&\siz ${}\in{}$&\siz $\{1,3\}$&\siz $(1,1,10,3)$&\siz &\siz $(4,4,10,\myfrac{9}{2})$&\siz \\ \hline \siz
      \tab\tab\tab\tab\tab\tab $c(10) = 1$&\siz $c(24)$&\siz ${}={}$&\siz $3$&\siz $12\cdot2 = 24$&\siz $(10,6,24,10)$&\siz $16\cdot\myfrac{3}{2} = 24$&\siz \\ \hline \siz
      \tab\tab\tab\tab\tab\tab\tab $c(24) = 3$&\siz $c(32)$&\siz ${}={}$&\siz $0$&\siz &\siz $(6,2,32,10)$&\siz $16\cdot2 = 32$&\siz $24\cdot\myfrac{4}{3} = 32$\\ \hline \siz
      \tab\tab\tab\tab\tab\tab\tab\tab $c(32) = 0$&\siz $c(11)$&\siz ${}={}$&\siz $2$&\siz $(32,1,11,11)$&\siz $(2,11,11,6)$&\siz &\siz $(24,9,11,11)$\\ \hline \siz
      \tab\tab\tab\tab\tab\tab\tab\tab\tab $c(11) = 2$&\siz $c(5)$&\siz ${}={}$&\siz $3$&\siz $(12,3,5,5)$&\siz $10\cdot\myfrac{1}{2} = 5$&\siz $(11,\myfrac{9}{2},\myfrac{9}{2},5)$&\siz \\ \hline \siz
      \tab\tab\tab\tab\tab\tab\tab\tab\tab\tab $c(5) = 3$&\siz $c(\myfrac{11}{2})$&\siz !?&\siz $ $&\siz $(1,\myfrac{11}{2},\myfrac{11}{2},3)$&\siz $(10,10,2,\myfrac{11}{2})$&\siz $11\cdot\myfrac{1}{2} = \myfrac{11}{2}$&\siz $(5,9,8,\myfrac{11}{2})$\\ \hline \siz
      \tab\tab\tab\tab\tab\tab $c(10) = 3$&\siz $c(24)$&\siz ${}={}$&\siz $1$&\siz $12\cdot2 = 24$&\siz &\siz $16\cdot\myfrac{3}{2} = 24$&\siz $(8,8,24,10)$\\ \hline \siz
      \tab\tab\tab\tab\tab\tab\tab $c(24) = 1$&\siz $c(7)$&\siz ${}={}$&\siz $0$&\siz &\siz $(24,2,2,7)$&\siz $(\myfrac{9}{2},\myfrac{9}{2},7,4)$&\siz $(10,10,8,7)$\\ \hline \siz
      \tab\tab\tab\tab\tab\tab\tab\tab $c(7) = 0$&\siz $c(18)$&\siz ${}={}$&\siz $2$&\siz $12\cdot\myfrac{3}{2} = 18$&\siz $24\cdot\myfrac{3}{4} = 18$&\siz &\siz $9\cdot2 = 18$\\ \hline \siz
      \tab\tab\tab\tab\tab\tab\tab\tab\tab $c(18) = 2$&\siz $c(\myfrac{27}{2})$&\siz !?&\siz $ $&\siz $(1,\myfrac{27}{2},\myfrac{27}{2},7)$&\siz $(24,24,6,\myfrac{27}{2})$&\siz $18\cdot\myfrac{3}{4} = \myfrac{27}{2}$&\siz $9\cdot\myfrac{3}{2} = \myfrac{27}{2}$\\ \hline \siz
      \tab\tab $c(12) = 2$&\siz $c(16)$&\siz ${}={}$&\siz $0$&\siz &\siz $(6,2,16,6)$&\siz $12\cdot\myfrac{4}{3} = 16$&\siz $8\cdot2 = 16$\\ \hline \siz
      \tab\tab\tab $c(16) = 0$&\siz $c(5)$&\siz ${}\in{}$&\siz $\{1,3\}$&\siz $(16,3,1,5)$&\siz &\siz $(12,4,4,5)$&\siz \\ \hline \siz
      \tab\tab\tab\tab $c(5) = 1$&\siz $c(7)$&\siz ${}\in{}$&\siz $\{0,3\}$&\siz &\siz $(6,7,7,5)$&\siz $(12,12,4,7)$&\siz \\ \hline \siz
      \tab\tab\tab\tab\tab $c(7) = 0$&\siz $c(9)$&\siz ${}={}$&\siz $3$&\siz $(16,3,9,7)$&\siz $6\cdot\myfrac{3}{2} = 9$&\siz $12\cdot\myfrac{3}{4} = 9$&\siz \\ \hline \siz
      \tab\tab\tab\tab\tab\tab $c(9) = 3$&\siz $c(14)$&\siz ${}={}$&\siz $2$&\siz $7\cdot2 = 14$&\siz $(5,5,14,6)$&\siz &\siz $(9,9,14,8)$\\ \hline \siz
      \tab\tab\tab\tab\tab\tab\tab $c(14) = 2$&\siz $c(20)$&\siz !?&\siz $ $&\siz $(7,1,20,7)$&\siz $(2,2,20,6)$&\siz $(14,14,20,12)$&\siz $(8,8,20,9)$\\ \hline \siz
      \tab\tab\tab\tab\tab $c(7) = 3$&\siz $c(10)$&\siz ${}={}$&\siz $2$&\siz $(1,1,10,3)$&\siz $5\cdot2 = 10$&\siz &\siz $(8,10,10,7)$\\ \hline \siz
      \tab\tab\tab\tab\tab\tab $c(10) = 2$&\siz $c(14)$&\siz ${}={}$&\siz $0$&\siz &\siz $(5,5,14,6)$&\siz $(12,14,14,10)$&\siz $7\cdot2 = 14$\\ \hline \siz
      \tab\tab\tab\tab\tab\tab\tab $c(14) = 0$&\siz $c(20)$&\siz ${}={}$&\siz $3$&\siz $(16,20,20,14)$&\siz $(2,2,20,6)$&\siz $10\cdot2 = 20$&\siz \\ \hline \siz
      \tab\tab\tab\tab\tab\tab\tab\tab $c(20) = 3$&\siz $c(9)$&\siz ${}={}$&\siz $0$&\siz &\siz $6\cdot\myfrac{3}{2} = 9$&\siz $12\cdot\myfrac{3}{4} = 9$&\siz $(20,7,9,9)$\\ \hline \siz
      \tab\tab\tab\tab\tab\tab\tab\tab\tab $c(9) = 0$&\siz $c(13)$&\siz ${}={}$&\siz $2$&\siz $(9,14,13,9)$&\siz $(5,6,13,6)$&\siz &\siz $(7,8,13,7)$\\ \hline \siz
      \tab\tab\tab\tab\tab\tab\tab\tab\tab\tab $c(13) = 2$&\siz $c(17)$&\siz !?&\siz $ $&\siz $(16,3,17,9)$&\siz $(5,2,17,6)$&\siz $(13,10,17,10)$&\siz $(7,8,17,8)$\\ \hline \siz
      \tab\tab\tab\tab $c(5) = 3$&\siz $c(\myfrac{9}{2})$&\siz ${}={}$&\siz $2$&\siz $3\cdot\myfrac{3}{2} = \myfrac{9}{2}$&\siz $6\cdot\myfrac{3}{4} = \myfrac{9}{2}$&\siz &\siz $(5,5,8,\myfrac{9}{2})$\\ \hline \siz
      \tab\tab\tab\tab\tab $c(\myfrac{9}{2}) = 2$&\siz $c(10)$&\siz ${}={}$&\siz $1$&\siz $(1,1,10,3)$&\siz &\siz $(4,4,10,\myfrac{9}{2})$&\siz $5\cdot2 = 10$\\ \hline \siz
      \tab\tab\tab\tab\tab\tab $c(10) = 1$&\siz $c(7)$&\siz ${}={}$&\siz $0$&\siz &\siz $(10,7,7,6)$&\siz $(\myfrac{9}{2},\myfrac{9}{2},7,4)$&\siz $(5,8,7,5)$\\ \hline \siz
      \tab\tab\tab\tab\tab\tab\tab $c(7) = 0$&\siz $c(\myfrac{15}{2})$&\siz !?&\siz $ $&\siz $(7,7,16,\myfrac{15}{2})$&\siz $10\cdot\myfrac{3}{4} = \myfrac{15}{2}$&\siz $(\myfrac{9}{2},4,\myfrac{15}{2},4)$&\siz $5\cdot\myfrac{3}{2} = \myfrac{15}{2}$\\ \hline \siz
      $c(6) = 3$&\siz $c(8)$&\siz ${}={}$&\siz $1$&\siz $(3,1,8,3)$&\siz &\siz $4\cdot2 = 8$&\siz $6\cdot\myfrac{4}{3} = 8$\\ \hline \siz
      \tab $c(8) = 1$&\siz $c(\myfrac{9}{2})$&\siz ${}={}$&\siz $2$&\siz $3\cdot\myfrac{3}{2} = \myfrac{9}{2}$&\siz $(8,8,2,\myfrac{9}{2})$&\siz &\siz $6\cdot\myfrac{3}{4} = \myfrac{9}{2}$\\ \hline \siz
      \tab\tab $c(\myfrac{9}{2}) = 2$&\siz $c(9)$&\siz ${}\in{}$&\siz $\{0,1\}$&\siz &\siz &\siz $\myfrac{9}{2}\cdot2 = 9$&\siz $6\cdot\myfrac{3}{2} = 9$\\ \hline \siz
      \tab\tab\tab $c(9) = 0$&\siz $c(12)$&\siz ${}={}$&\siz $2$&\siz $9\cdot\myfrac{4}{3} = 12$&\siz $8\cdot\myfrac{3}{2} = 12$&\siz &\siz $6\cdot2 = 12$\\ \hline \siz
      \tab\tab\tab\tab $c(12) = 2$&\siz $c(\myfrac{3}{2})$&\siz ${}={}$&\siz $3$&\siz $3\cdot\myfrac{1}{2} = \myfrac{3}{2}$&\siz $2\cdot\myfrac{3}{4} = \myfrac{3}{2}$&\siz $(12,\myfrac{9}{2},\myfrac{3}{2},\myfrac{9}{2})$&\siz \\ \hline \siz
      \tab\tab\tab\tab\tab $c(\myfrac{3}{2}) = 3$&\siz $c(\myfrac{9}{4})$&\siz ${}={}$&\siz $1$&\siz $3\cdot\myfrac{3}{4} = \myfrac{9}{4}$&\siz &\siz $\myfrac{9}{2}\cdot\myfrac{1}{2} = \myfrac{9}{4}$&\siz $\myfrac{3}{2}\cdot\myfrac{3}{2} = \myfrac{9}{4}$\\ \hline \siz
      \tab\tab\tab\tab\tab\tab $c(\myfrac{9}{4}) = 1$&\siz $c(\myfrac{27}{8})$&\siz ${}={}$&\siz $0$&\siz &\siz $\myfrac{9}{4}\cdot\myfrac{3}{2} = \myfrac{27}{8}$&\siz $\myfrac{9}{2}\cdot\myfrac{3}{4} = \myfrac{27}{8}$&\siz $(\myfrac{3}{2},6,6,\myfrac{27}{8})$\\ \hline \siz
      \tab\tab\tab\tab\tab\tab\tab $c(\myfrac{27}{8}) = 0$&\siz $c(\myfrac{9}{8})$&\siz ${}={}$&\siz $2$&\siz $(\myfrac{27}{8},9,\myfrac{9}{8},\myfrac{27}{8})$&\siz $\myfrac{9}{4}\cdot\myfrac{1}{2} = \myfrac{9}{8}$&\siz &\siz $\myfrac{3}{2}\cdot\myfrac{3}{4} = \myfrac{9}{8}$\\ \hline \siz
      \tab\tab\tab\tab\tab\tab\tab\tab $c(\myfrac{9}{8}) = 2$&\siz $c(\myfrac{3}{4})$&\siz ${}={}$&\siz $1$&\siz $1\cdot\myfrac{3}{4} = \myfrac{3}{4}$&\siz &\siz $\myfrac{9}{8}\cdot\myfrac{2}{3} = \myfrac{3}{4}$&\siz $\myfrac{3}{2}\cdot\myfrac{1}{2} = \myfrac{3}{4}$\\ \hline \siz
      \tab\tab\tab\tab\tab\tab\tab\tab\tab $c(\myfrac{3}{4}) = 1$&\siz $c(\myfrac{15}{2})$&\siz ${}={}$&\siz $3$&\siz $(3,3,\myfrac{15}{2},\myfrac{27}{8})$&\siz $(\myfrac{3}{4},\myfrac{3}{4},\myfrac{15}{2},\myfrac{9}{4})$&\siz $(\myfrac{9}{2},4,\myfrac{15}{2},4)$&\siz \\ \hline \siz
      \tab\tab\tab\tab\tab\tab\tab\tab\tab\tab $c(\myfrac{15}{2}) = 3$&\siz $c(\myfrac{21}{4})$&\siz !?&\siz $ $&\siz $(\myfrac{27}{8},\myfrac{27}{8},\myfrac{21}{4},3)$&\siz $(\myfrac{3}{4},2,\myfrac{21}{4},2)$&\siz $(12,\myfrac{9}{2},\myfrac{9}{2},\myfrac{21}{4})$&\siz $(\myfrac{15}{2},\myfrac{15}{2},6,\myfrac{21}{4})$\\ \hline \siz
      \tab\tab\tab $c(9) = 1$&\siz $c(10)$&\siz ${}\in{}$&\siz $\{1,3\}$&\siz $(1,1,10,3)$&\siz &\siz $(4,4,10,\myfrac{9}{2})$&\siz \\ \hline \siz
      \tab\tab\tab\tab $c(10) = 1$&\siz $c(7)$&\siz ${}\in{}$&\siz $\{0,3\}$&\siz &\siz $(10,10,8,7)$&\siz $(\myfrac{9}{2},\myfrac{9}{2},7,4)$&\siz \\ \hline \siz
      \tab\tab\tab\tab\tab $c(7) = 0$&\siz $c(\myfrac{8}{3})$&\siz ${}={}$&\siz $3$&\siz $(7,1,\myfrac{8}{3},\myfrac{8}{3})$&\siz $2\cdot\myfrac{4}{3} = \myfrac{8}{3}$&\siz $4\cdot\myfrac{2}{3} = \myfrac{8}{3}$&\siz \\ \hline \siz
      \tab\tab\tab\tab\tab\tab $c(\myfrac{8}{3}) = 3$&\siz $c(\myfrac{4}{3})$&\siz ${}={}$&\siz $2$&\siz $1\cdot\myfrac{4}{3} = \myfrac{4}{3}$&\siz $2\cdot\myfrac{2}{3} = \myfrac{4}{3}$&\siz &\siz $\myfrac{8}{3}\cdot\myfrac{1}{2} = \myfrac{4}{3}$\\ \hline \siz
      \tab\tab\tab\tab\tab\tab\tab $c(\myfrac{4}{3}) = 2$&\siz $c(\myfrac{16}{3})$&\siz ${}={}$&\siz $0$&\siz &\siz $8\cdot\myfrac{2}{3} = \myfrac{16}{3}$&\siz $4\cdot\myfrac{4}{3} = \myfrac{16}{3}$&\siz $\myfrac{8}{3}\cdot2 = \myfrac{16}{3}$\\ \hline \siz
      \tab\tab\tab\tab\tab\tab\tab\tab $c(\myfrac{16}{3}) = 0$&\siz $c(\myfrac{32}{3})$&\siz !?&\siz $ $&\siz $\myfrac{16}{3}\cdot2 = \myfrac{32}{3}$&\siz $8\cdot\myfrac{4}{3} = \myfrac{32}{3}$&\siz $(\myfrac{4}{3},4,\myfrac{32}{3},4)$&\siz $(\myfrac{8}{3},\myfrac{32}{3},\myfrac{32}{3},6)$\\ \hline \siz
      \tab\tab\tab\tab\tab $c(7) = 3$&\siz $c(5)$&\siz ${}\in{}$&\siz $\{0,2\}$&\siz &\siz $10\cdot\myfrac{1}{2} = 5$&\siz &\siz $(7,7,6,5)$\\ \hline \siz
      \tab\tab\tab\tab\tab\tab $c(5) = 0$&\siz $c(12)$&\siz ${}={}$&\siz $2$&\siz $(5,3,12,5)$&\siz $9\cdot\myfrac{4}{3} = 12$&\siz &\siz $6\cdot2 = 12$\\ \hline \siz
      \tab\tab\tab\tab\tab\tab\tab $c(12) = 2$&\siz $c(16)$&\siz !?&\siz $ $&\siz $(3,1,16,5)$&\siz $8\cdot2 = 16$&\siz $12\cdot\myfrac{4}{3} = 16$&\siz $(6,6,16,7)$\\ \hline \siz
      \tab\tab\tab\tab\tab\tab $c(5) = 2$&\siz $c(12)$&\siz ${}={}$&\siz $0$&\siz &\siz $9\cdot\myfrac{4}{3} = 12$&\siz $(4,4,12,5)$&\siz $6\cdot2 = 12$\\ \hline \siz
      \tab\tab\tab\tab\tab\tab\tab $c(12) = 0$&\siz $c(\myfrac{7}{2})$&\siz !?&\siz $ $&\siz $(12,1,1,\myfrac{7}{2})$&\siz $(10,2,2,\myfrac{7}{2})$&\siz $(5,5,4,\myfrac{7}{2})$&\siz $7\cdot\myfrac{1}{2} = \myfrac{7}{2}$\\ \hline \siz
      \tab\tab\tab\tab $c(10) = 3$&\siz $c(\myfrac{16}{3})$&\siz ${}={}$&\siz $0$&\siz &\siz $8\cdot\myfrac{2}{3} = \myfrac{16}{3}$&\siz $4\cdot\myfrac{4}{3} = \myfrac{16}{3}$&\siz $(10,6,\myfrac{16}{3},\myfrac{16}{3})$\\ \hline \siz
      \tab\tab\tab\tab\tab $c(\myfrac{16}{3}) = 0$&\siz $c(\myfrac{8}{3})$&\siz ${}={}$&\siz $3$&\siz $\myfrac{16}{3}\cdot\myfrac{1}{2} = \myfrac{8}{3}$&\siz $2\cdot\myfrac{4}{3} = \myfrac{8}{3}$&\siz $4\cdot\myfrac{2}{3} = \myfrac{8}{3}$&\siz \\ \hline \siz
      \tab\tab\tab\tab\tab\tab $c(\myfrac{8}{3}) = 3$&\siz $c(\myfrac{4}{3})$&\siz ${}={}$&\siz $2$&\siz $1\cdot\myfrac{4}{3} = \myfrac{4}{3}$&\siz $2\cdot\myfrac{2}{3} = \myfrac{4}{3}$&\siz &\siz $\myfrac{8}{3}\cdot\myfrac{1}{2} = \myfrac{4}{3}$\\ \hline \siz
      \tab\tab\tab\tab\tab\tab\tab $c(\myfrac{4}{3}) = 2$&\siz $c(\myfrac{32}{3})$&\siz !?&\siz $ $&\siz $\myfrac{16}{3}\cdot2 = \myfrac{32}{3}$&\siz $8\cdot\myfrac{4}{3} = \myfrac{32}{3}$&\siz $(\myfrac{4}{3},4,\myfrac{32}{3},4)$&\siz $(\myfrac{8}{3},\myfrac{32}{3},\myfrac{32}{3},6)$\\ \hline
    \end{tabular}
    \caption{The computer-generated proof of Lemma~\ref{lemma:forb2}.}
    \label{table:proof}
  \end{table}
  \normalsize

  We present a computer-generated proof of the rest of this lemma in
  Table~\ref{table:proof}.  Without loss of generality, assume that our
  set of colors is $\{0,1,2,3\}$ and that $c(1)=c(3)=0$, $c(2)=1$, and
  $c(4)=2$; we must derive a contradiction.

  We describe briefly how to read Table~\ref{table:proof}.  In the
  left-most column, the assumptions that we make are listed.  They are
  structured in a tree-like fashion, reflecting the trial-and-error
  argument.  The next column lists the current claim.  If this claim is
  ``$c(x)$!?'', then this means that $c(x)$ cannot be colored with any
  color and a contradiction has been obtained (thus we must backtrack);
  if the claim is $c(x) = a$, then this will be the next assumption.  If
  the claim is $c(x) \in \{a,b\}$, then we must consider the cases $c(x)
  = a$ and $c(x) = b$ separately.  Finally, the last four columns
  describe why $c(x) \ne 0$, \emph{etc.}, if this is needed to support
  the claim. An equation of the form $y\cdot r = x$ means that the
  forbidden ratio $r$ forbids $x$ and $y$ from being the same color; a
  $4$-tuple $(x_1, x_2, x_3, x_4)$ means that if we colored $x$ the
  color in question, then this would be a monochromatic solution to $x_1
  + x_2 + x_3 = 4 x_4$.

  For example, the first line of the proof can be read as follows:
  ``$c(6)$ is either $1$ or $3$ because of forbidden ratios from $3$ and
  $4$; we consider these possibilities separately.''
\end{proof}

\begin{proof}[Proof of Lemma~\ref{lemma:forb3}.]
  For the sake of contradiction, assume that $c(x)=c(4x)$ for some $x$.
  As in the previous lemma, we will use forbidden ratios; here, they are
  $2$, $\frac 32$, $\frac 43$, and $3$.  To begin, note that $c(x) =
  c(4x)$, $c(2x)$, $c(3x)$, and $c(6x)$ must be different colors by
  these forbidden ratios.  Now, $c(\frac 32x) = c(6x)$ by forbidden
  ratios from $x$, $2x$, and $3x$.  It follows that $c(\frac 94x) =
  c(2x)$ by forbidden ratios from $3x$ and $\frac 32x$ as well as the
  solution $(4x,4x,x,\frac 94x)$.  Then $c(\frac 92x) = c(x)$ by
  forbidden ratios from $\frac 94x$, $3x$, and $\frac 32x$.  Next,
  $c(9x) = c(2x)$ by forbidden ratios from $\frac 92x$, $3x$, and $6x$.
  Further, $c(12x)=c(3x)$ by forbidden ratios from $4x$, $9x$, and $6x$.
  Also, $c(8x) = c(2x)$ by forbidden ratios from $4x$, $12x$, and $6x$.
  Nearing the end, $c(\frac 12x) = c(3x)$ by forbidden ratios from $x$
  and $\frac 32x$ as well as the solution $(8x,\frac 12x,\frac 12x,
  \frac94x)$.  Finally, there are no possibilities left for $c(\frac
  34x)$ by forbidden ratios from $x$, $\frac 94x$, $\frac 12x$, and
  $\frac 32x$.
\end{proof}

\begin{proof}[Proof of Lemma~\ref{lemma:235-2}.]
By induction, it suffices to prove that $c(x)=c(6x)=c(16x)$ for every
nonzero rational number $x$. We have $x$, $2x$, $3x$, and $4x$ are all
different colors by the previous three lemmas. Also, $2x$, $3x$, $4x$,
and $6x$ are all different colors by Lemma~\ref{lemma:forb1} and
Lemma~\ref{lemma:forb2}. Hence, $c(x)=c(6x)$. By the previous three
lemmas, the numbers $2x$, $4x$, $6x$, and $8x$ are all different
colors. Hence, $c(8x)=c(3x)$. By the previous three lemmas, $3x$, $4x$,
$6x$, and $12x$ are all different colors, so $c(12x)=c(2x)$. By the
previous three lemmas, $4x$, $8x$, $12x$, and $16x$ are all different
colors, so $c(16x)=c(x)$, completing the proof.
\end{proof}

\begin{proof}[Proof of Lemma~\ref{lemma:-4}.]
Since $c(x)=c(6x)$ and $(-x,-x,6x,x)$ is a solution to
$x_1+x_2+x_3=4x_4$, then $-x$ and $x$ are different colors. Since
$c(2x)=c(\frac{3}{4}x)$ and $(-x,2x,2x,\frac{3}{4}x)$ is a solution to
$x_1+x_2+x_3=4x_4$, then $-x$ and $2x$ are different colors.

By Lemma~\ref{lemma:235-2}, $5x$, $10x$, $15x$, and $20x$ are all
different colors, so $c(3x) \in \{5x,\penalty0 10x,\penalty0
15x,\penalty0 20x\}$. By Lemma~\ref{lemma:235-2}, we have
$c(3x)=c(8x)=c(18x)=c(\frac{27}{4}x)$. Since $(5x,-x,8x,3x)$,
$(10x,-x,3x,3x)$, $(15x,-x,18x,8x)$, and $(20x,-x,8x,\frac{27}{4}x)$ are
solutions to $x_1+x_2+x_3=4x_4$, then $-x$ and $3x$ are different
colors. Hence, $c(-x)=c(4x)$.
\end{proof}

\begin{proof}[Proof of Lemma~\ref{lemma:mod5}.]
By Lemma~\ref{lemma:235-2}, for integers $m$ and $n$ and nonzero
rational number $x$, we have $c(2^{m}3^{n}x)=c(x)$ if and only if
$w_5(2^m3^n)\equiv 1 \pmod{5}$. By Lemma~\ref{lemma:-4}, we have
$c(-x)=c(4x)$ for all nonzero rational numbers $x$. By induction, it
suffices to prove for every positive integer $p$ that is not a multiple
of $5$, we have, $c(px)=c(dx)$ for $d \in \{1,2,3,4\}$ if and only if $p
\equiv d \pmod{5}$. By Lemma~\ref{lemma:iff3multiple2}, we have already
established this for $p=2^m3^n$ and $n$ and $m$ are integers.

The proof is by induction on the size of $p$. Suppose $p>5$ is an
integer that is not a multiple of $2$, $3$, or $5$. We write $p=10a+b$,
where $a$ and $b$ are nonnegative integers and $b \in \{1,3,7,9\}$. The
induction hypothesis is that, for $d \in \{1,2,3,4\}$, we have
$c(p'q)=c(dq)$ for every nonzero rational number $q$ and positive
integer $p'$ such that $p' \equiv d \pmod{5}$, $p'<p$, and $p'$ is not a
multiple of $2$, $3$, or $5$.

By Lemma~\ref{lemma:235-2}, the four colors $c(5x), c(10x), c(15x),
c(20x)$ are distinct for every nonzero rational number $x$. Hence,
$c(x), c(2x), c(3x), c(4x) \in \{c(5x), c(10x), c(15x), c(20x)\}$ for
every nonzero rational number $x$.

\textbf{Case 1:} $p=10a+1$ with $a \geq 1$.

We have $c(3x)=c(8x)$.  Since $0<5a+6<p$, then by the induction
hypothesis, we have $c(3x)=c((\frac{5a+6}{2})x)$. Since
$(px,3x,8x,(\frac{5a+6}{2})x)$ is a solution to $x_1+x_2+x_3=4x_4$, then
$px$ and $3x$ are different colors.

We have $c(14x)=c(4x)$. Since $0<5a+4<p$, then by the induction
hypothesis, we have $c(4x)=c((5a+4)x)$. Since $(px,px,14x,(5a+4)x)$ is a
solution to $x_1+x_2+x_3=4x_4$, then $px$ and $4x$ are different colors.

By the induction hypothesis, Lemma~\ref{lemma:235-2}, and
Lemma~\ref{lemma:mod5}, we have $c(-3x)=c(2x)=c(7x)=c(12x)$. Since
$0<5a+9<p$ for $a>1$ and $\frac{5a+9}{2}=7<11$ for $a=1$, then by the
induction hypothesis, we have $c((\frac{5a+9}{2})x)=c(2x)$. Since
$(5x,12x,px,(\frac{5a+9}{2})x)$, $(10x,7x,px, (\frac{5a+9}{2})x)$,
$(15x,2x,px,(\frac{5a+9}{2})x)$, $(20x,-3x,px,(\frac{5a+9}{2})x)$ are
solutions to $x_1+x_2+x_3=4x_4$ and $c(2x) \in
\{c(5x),c(10x),c(15x),c(20x)\}$, then $px$ and $2x$ are different
colors. Hence, $c(px)=c(x)$.

\textbf{Case 2:} $p=10a+3$ with $a \geq 1$.

Since $0<5a+2<p$, then by the induction hypothesis, we have $(5a+2)x$
and $2x$ are the same color. Since $(px,px,2x,(5a+2)x)$ is a solution to
$x_1+x_2+x_3=4x_4$, then $px$ and $2x$ are different colors.

Since $p=10a+3 \geq 13$, then by the induction hypothesis, we have
$c(4x)=c(9x)=c(14x)$. For $a=1$, we have $p=13$, and $(13x,14x,9x,9x)$
is a solution to $x_1+x_2+x_3=4x_4$, so $13x$ and $4x$ are different
colors. For $a>1$, we have $0<5a+8 < p$ , so by the induction
hypothesis, we have $(\frac{5a+8}{2})x$ and $4x$ are the same
color. Since $(px,9x,4x,(\frac{5a+8}{2})x)$ is a solution to
$x_1+x_2+x_3=4x_4$, then $px$ and $4x$ are different colors.

Since $p=10a+3 \geq 13$ and $0<5a+7<p$, then by the induction
hypothesis, we have
$c(x)=c(6x)=c(\frac{7}{2}x)=c(\frac{72}{7}x)=c((\frac{5a+7}{2})x)$,
$c(\frac{5}{7}x)=c(15x)$, and $c(\frac{15}{2}x)=c(20x)$. Since the
tuples $(5x,px,6x,(\frac{5a+7}{2})x)$, $(10x,px,x,(\frac{5a+7}{2})x)$,
$(\frac{5}{7}x,px,\frac{72}{7}x,(\frac{5a+7}{2})x)$, and
$(\frac{15}{2}x,px,\frac{7}{2}x,(\frac{5a+7}{2})x)$ are solutions to
$x_1+x_2+x_3=4x_4$, and $c(x) \in \{c(5x),\penalty0 c(10x),\penalty0
c(15x),\penalty0 c(20x)\}$, then $px$ and $x$ are different
colors. Hence, $c(px)=c(3x)$.

\textbf{Case 3:} $p=10a+7$ with $a \geq 0$.

We have $c(-4x)=c(x)$. Since $0<5a+2<10a+7$, then by the induction
hypothesis, we have $c((\frac{5a+2}{2})x)=c(x)$. Since
$(px,-4x,x,(\frac{5a+2}{2})x)$ is a solution to $x_1+x_2+x_3=4x_4$, then
$px$ and $x$ are different colors.

We have $c(-2x)=c(3x)$. Since $0<5a+3<10a+7$, then by the induction
hypothesis, we have $c(5a+3)x)=c(3x)$. Since $(px,px,-2x,(5a+3)x)$ is a
solution to $x_1+x_2+x_3=4x_4$, then $px$ and $3x$ are different colors.

Note that $c(-16x)=c(-6x)=c(-\frac{9}{4}x)=c(-\frac{8}{3}x)=c(4x)$,
$c(\frac{5}{3}x)=c(10x)$, and $c(\frac{5}{4}x)=c(20x)$. Since
$0<5a+3<10a+7$, then by the induction hypothesis, we have
$(\frac{5a+3}{2})x$ and $4x$ are the same color. Since
$(5x,px,-6x,(\frac{5a+3}{2})x)$,
$(\frac{5}{3}x,px,-\frac{8}{3}x,(\frac{5a+3}{2})x)$,
$(15x,px,-16x,(\frac{5a+3}{2})x)$, and
$(\frac{5}{4}x,px,-\frac{9}{4}x,(\frac{5a+3}{2})x)$ are solutions to
$x_1+x_2+x_3=4x_4$ and $c(4x) \in \{c(5x),\penalty0 c(10x),\penalty0
c(15x),\penalty0 c(20x)\}$, then $px$ and $4x$ are different colors.
Hence, $c(px)=c(2x)$.

\textbf{Case 4:} $p=10a+9$ with $a \geq 0$.

We have $c(x)=c(6x)$. Since $0<5a+6<p$, then by the induction
hypothesis, we have $(5a+6)x$ and $x$ are the same color. Since
$(px,px,6x,(5a+6)x)$ is a solution to $x_1+x_2+x_3=4x_4$, then $px$ and
$x$ are different colors.

We have $c(-3x)=c(2x)$. Since $0<5a+4<p$, then by the induction
hypothesis, we have $(\frac{5a+4}{2})x$ and $2x$ are the same
color. Since $(px,2x,-3x,(\frac{5a+4}{2})x)$ is a solution to
$x_1+x_2+x_3=4x_4$, then $px$ and $2x$ are different colors.

Since $p=10a+9 \geq 19$, then by the induction hypothesis, we have
$c(3)=c(-2)=c(-7)=c(-12)=c(-17)$. Since $0<5a+6<p$, then
$(\frac{5a+6}{2})x$ and $3x$ are the same color. Since
$(5x,px,-2x,(\frac{5a+6}{2})x)$, $(10x,px,-7x,(\frac{5a+6}{2})x)$,
$(15x,px,-12x,(\frac{5a+6}{2})x)$, and $(20x,px,-17x,(\frac{5a+6}{2})x)$
are solutions to $x_1+x_2+x_3=4x_4$, then $px$ and $3x$ are different
colors.  Hence, $px$ and $4x$ are the same color.
\end{proof}

\section{Minimal colorings of the nonzero reals}
\label{section:real} Let $A=\{a_1,\dotsc,a_n\}$ be a nonempty finite
multiset of nonzero real numbers. For example, $\{2,3\}$ and $\{2,2,3\}$
are considered to be different multisets. We define a coloring $c$ to be
\emph{strongly free} of monochromatic solutions to $a_1x_1+\dotsb+a_nx_n
= 0$ if for every nonempty subset $A' \subseteq A$, the coloring $c$ has
no monochromatic solutions to the equation $E(A')$ defined by
\[ E(A')\colon \sum_{a_i \in A'} a_ix_i=0. \]

\begin{proposition}
Suppose $c\colon \mathbb{Q}\setminus \{0\}\rightarrow \{1,\dotsc,r\}$ is
an $r$-coloring of the nonzero rational numbers with $r \not =1$ a
positive integer, and $c$ is \emph{strongly free} of monochromatic
solutions to $a_1x_1+\dotsb+a_nx_n=0$. Then in ZFC, there exists
$2^{2^{\aleph_0}}$ $r$-colorings of the nonzero real numbers that are
strongly free of monochromatic solution to $a_1x_1+\dotsb+a_nx_n=0$.
\end{proposition}
\begin{proof}
In ZFC, every vector space has a basis. Viewing $\mathbb{R}$ as a
$\mathbb{Q}$-vector space, there is a well-ordered basis
$B=\{b_j\}_{j<2^{\omega}}$ for $\mathbb{R}$ as a $\mathbb{Q}$-vector
space. So every real number $x$ has a unique representation as
$x=\sum_{j<2^{\omega}} q_jb_j$, where each $q_j$ is rational and $q_j
\ne 0$ for only finitely many $j$. For a nonzero real number $x$, let
$j(x)$ be the least ordinal $j$ such that $q_j$ is nonzero. For a finite
set $S$ of nonzero real numbers, let $j(S)=\min_{x \in S}j(x)$.

Let $\Pi=(\pi_j)_{j<2^{\omega}}$ be any sequence of permutations $\pi_j$
of the set $\{1,\dotsc,r\}$ with $\pi_0$ being the identity
permutation. Define the coloring $C_{\Pi}$ by
$C_{\Pi}(x)=\pi_{j(x)}(c(q_{j(x)}))$.

Suppose $x_{1},\dotsc,x_{n}$ are nonzero real numbers and $A' \subseteq
A$ is a nonempty subset such that $\sum_{a_i \in A'} a_ix_i=0$. Let
$q_{i,j}$ be the coefficient of $b_j$ in the representation of
$x_i$. Since $\sum_{a_i \in A'} a_ix_i=0$, then $\sum_{a_i \in A'}
a_iq_{i,j(A')}=0$. Letting $A'' \subset A'$ be those $a_i \in A'$ such
that $q_{i.j(A')} \ne 0$, we see that $A''$ is a nonempty subset of $A$
and $\sum_{a_i \in A''} a_iq_{i,j(A')}=0$. Since
$C_{\Pi}(x_i)=\pi_{j(A')}(c(q_{i,j(A')}))$ for $a_i \in A''$ and $c$ is
strongly free of monochromatic solutions to $a_1x_1+\dotsb+a_nx_n=0$,
then the set $\{x_i\}_{a_i \in A''}$ is not monochromatic and $C_{\Pi}$
is strongly free of monochromatic solutions to
$a_1x_1+\dotsb+a_nx_n=0$. Since there are $2^{2^{\aleph_0}}$
nonisomorphic $r$-colorings of the form $C_{\Pi}$ and there are a total
of $2^{2^{\aleph_0}}$ nonisomorphic $r$-colorings of the real numbers,
then there are $2^{2^{\aleph_0}}$ colorings of the nonzero real numbers
that are strongly free of monochromatic solutions to
$a_1x_1+\dotsb+a_nx_n=0$.
\end{proof}

For $p$ a prime number and $n \geq 3$, each of the colorings $c_{p,n}$
is strongly free of monochromatic solutions to $E(p,n)$, so there are
$2^{2^{\aleph_0}}$ different $n$-colorings of the nonzero real numbers
without a monochromatic solution to $E(p,n)$.  In general, define the
$n$-coloring $c_{p,v,n}\colon \mathbb{Q}\setminus\{0\} \rightarrow
\{0,\dotsc,n-1\}$ by $c_{p,v,n}(x)\equiv \bigl \lfloor \frac{v_p(x)}{v}
\bigr \rfloor \pmod{n}$.  Notice that if $q$ is a nonzero rational
number and $v_p(q) \ne 0$, then $c_{p,v,n}$ with $v=v_p(q)$ is strongly
free of monochromatic solutions to $E(q,n)$.

We now turn our attention to the ZF+LM system of axioms. We prove
Lemma~\ref{lemma:lm} using the following multiplicative version of a
theorem of Steinhaus \cite{St}. For a set $A$ of real numbers, define
$A/A=\{a/a':a,a' \in A, a' \ne 0\}$.

\begin{theorem}[Steinhaus's Theorem]
If $A$ is a set of real numbers with positive Lebesgue measure, then
$A/A$ contains an entire interval $(1-\epsilon,1+\epsilon)$ for some
$\epsilon>0$.
\end{theorem}

\begin{proof}[Proof of Lemma~\ref{lemma:lm}]
Suppose for contradiction that all color classes of $c$ are Lebesgue
measurable. Since $c$ partitions $\mathbb{R}$ into countably many color
classes, and Lebesgue measure is countably additive, then at least one
color class $C$ has positive Lebesgue measure. Since $x,da^mb^nx$ are
different colors for all integers $m$ and $n$, then $C/C$ and
$\{da^mb^n:m,n \in \mathbb{Z}\}$ are disjoint sets. Since $\log_a b$ is
irrational, then the set $\{da^mb^n:m,n \in \mathbb{Z}\}$ contains
numbers arbitrarily close to $1$. But by Steinhaus's theorem, $C/C$
contains an entire interval around $1$, so $C/C$ and $\{da^mb^n:m,b \in
\mathbb{Z}\}$ are not disjoint, a contradiction.
\end{proof}

\section{Acknowledgements}
The authors would like to thank Steven Butler, Daniel Kleitman, Rados
Radoi\v ci\'c, and Richard Stanley for interesting discussions
concerning the paper.

\bibliographystyle{amsalpha}
%\bibliography{references}
\providecommand{\bysame}{\leavevmode\hbox to3em{\hrulefill}\thinspace}
\providecommand{\MR}{\relax\ifhmode\unskip\space\fi MR }
\providecommand{\MRhref}[2]{%
  \href{http://www.ams.org/mathscinet-getitem?mr=#1}{#2}
}
\providecommand{\href}[2]{#2}
\renewcommand{\refname}{{\normalsize References}}

\end{document}